\documentclass[a4paper, 11pt]{article}

\usepackage[latin1]{inputenc}
\usepackage[english]{babel}
\usepackage{graphicx}
\usepackage{amssymb}
\usepackage{amsmath}
\usepackage{amsfonts}
\usepackage{amsthm}
\usepackage{mcode}

\newtheorem{theorem}{Theorem}

\newtheorem{corollary}{Corollary}

\newtheorem{conjecture}{Conjecture}

\newtheorem{definition}{Definition}
\newtheorem{example}{Example}
\newtheorem{lemma}{Lemma}

\newtheorem{question}{Question}

\newtheorem{remark}{Remark}
\numberwithin{equation}{section}

\DeclareMathOperator{\rank}{rank}
\DeclareMathOperator{\vertices}{vertices}
\DeclareMathOperator{\faces}{faces}
\DeclareMathOperator{\col}{col}
\DeclareMathOperator{\conv}{conv}
\def \matlab    {MATLAB$^{\text{\tiny \textregistered}}$ }

\title{On the Geometric Interpretation of the Nonnegative Rank}

\author{\normalsize Nicolas Gillis${}^1$ and Fran\c{c}ois Glineur${}^1$}
\date{}

\begin{document}

\maketitle

\begin{abstract} 
The nonnegative rank of a nonnegative matrix is the minimum number of nonnegative rank-one factors needed to reconstruct it exactly. The problem of determining this rank and computing the corresponding nonnegative factors is difficult; however it has many potential applications, e.g., in data mining, graph theory and computational geometry. In particular, it can be used to characterize the minimal size of any extended reformulation of a given combinatorial optimization program. 
In this paper, we introduce and study a related quantity, called the restricted nonnegative rank. We show that computing this quantity is equivalent to a problem in polyhedral combinatorics, and fully characterize its computational complexity. This in turn sheds new light on the nonnegative rank problem, and in particular allows us to provide new improved lower bounds based on its geometric interpretation. We apply these results to slack matrices and linear Euclidean distance matrices and obtain counter-examples to two conjectures of Beasly and Laffey, namely we show that the nonnegative rank of linear Euclidean distance matrices is not necessarily equal to their dimension, and that the rank of a matrix is not always greater than the nonnegative rank of its square.  
\bigskip

  \bigskip

\end{abstract}


\footnotetext[1] {Universit\'e catholique de Louvain, CORE, B-1348 Louvain-la-Neuve, Belgium. 
E-mail: nicolas.gillis@uclouvain.be and francois.\mbox{glineur}@uclouvain.be.  Nicolas Gillis is a research fellow of the Fonds de la Recherche Scientifique (F.R.S.-FNRS).  This text presents research results of the Belgian Program on Interuniversity Poles of Attraction initiated by the Belgian State, Prime Minister's Office, Science Policy
Programming. The scientific responsibility is assumed by the authors.}

\section{Introduction} \label{intro}

The nonnegative rank of a $m \times n$ real nonnegative matrix $M \in \mathbb{R}^{m \times n}_+$ is the minimum number of nonnegative rank-one factors needed to reconstruct $M$ exactly, i.e., the minimum $k$ such that there exists $U \in \mathbb{R}^{m \times k}_+$ and $V \in \mathbb{R}^{k \times n}_+$ with $M = UV = \sum_{i = 1}^k U_{:i} V_{i:}$. The pair $(U,V)$ is called a rank-$k$ nonnegative factorization\footnote{Notice that matrices $U$ and $V$ in a rank-$k$ nonnegative factorization are not required to have rank $k$.}  of $M$.  The nonnegative rank of $M$ is denoted $\rank_+(M)$. Clearly,
\begin{equation}
\rank(M) \leq \rank_+(M) \leq \min(m,n). \nonumber
\end{equation}
Determining the nonnegative rank and computing the corresponding nonnegative factorization is a relatively recently studied problem in linear algebra \cite{BP73, CR93}. In the literature, much more attention has been devoted to the approximate nonnegative factorization problem (called nonnegative matrix factorization, NMF for short \cite{LS99}) consisting in finding two low-rank nonnegative factors $U$ and $V$ such that $M \approx UV$ or, more precisely, solving \[ \min_{U \in \mathbb{R}^{m \times k}_+, V \in \mathbb{R}^{k \times n}_+} ||M - UV||_F. \]
NMF has been widely used as a data analysis technique \cite{BBLPP07}, e.g., in text mining, image processing, hyperspectral data analysis, computational biology, clustering, etc. Nevertheless, there are not too many theoretical results about the nonnegative rank and better characterizations, in particular lower bounds, could help practitioners. For example, efficient computations of nonnegative factorizations could help to design new NMF algorithms using a two-step strategy \cite{Vav}: first approximate $M$ with a low-rank nonnegative matrix $A$ (e.g., using the singular value decomposition\footnote{Even though the optimal low-rank approximation of a nonnegative matrix might not necessarily be nonnegative (except in the rank-one case), it is often the case in practice \cite{KCD09}.}) 
and then compute a nonnegative factorization of $A$. Bounds for the nonnegative rank could also help select the factorization rank of the NMF, replacing the trial and error approach often used by practitioners. For example, in hyperspectral image analysis, the nonnegative rank corresponds to the number of materials present in the image and its computation could lead to more efficient algorithms detecting these constitutive elements, see  \cite{B94, C94, IC99, GP10} and references therein.  \\


An extended formulation (or lifting) for a polytope $P \subset \mathbb{R}^n$ is a polyhedron $Q \subset \mathbb{R}^{n+p}$ 
such that 
\[
P \; =  \; \text{proj}_x(Q) \; {:=} \; \{ x \in \mathbb{R}^{n} \, | \, \exists y \in \mathbb{R}^p \text{ s.t. } (x,y) \in Q \}. 
\] 
Extended formulations whose size (number of constraints plus number of variables defining $Q$) is polynomial in $n$ are called \emph{compact} and are of great importance in integer programming. They allow to reduce significantly the size of the linear programming (LP) formulation of certain integer programs, and therefore provide a way to solve them efficiently, i.e., in polynomial-time (see \cite{CCCZ09} for a survey).  
Yannakakis \cite[Theorem 3]{Y91} showed that the minimum size $s$ of an extended formulation of a polytope\footnote{This can be generalized to polyhedra \cite{CCCZ09}.} 
\[
P = \{ x \in \mathbb{R}^n \, | \, Cx \geq d, Ax = b\}, 
\] 
is of the same order as the sum of its dimension $n$ and the nonnegative rank of its slack matrix $S_M \geq 0$, where each column of the slack matrix is defined as 
\begin{equation} \label{SMdef}
S_M(:,i) = Cv_i - d \geq 0, \quad i = 1, 2, \dots, m,
\end{equation}
and vectors $v_i$ are the $m$ vertices of the polytope $P$. Formally, we then have 
\[
s = \Theta(n + \rank_+(S_M)). 
\]
In particular, any rank-$k$ nonnegative factorization $(U,V)$ of $S_M = UV$ provides the following extended formulation for $P$ with size $\Theta(n + k)$  
\begin{equation} \label{extQ} 
Q = \{ (x,y) \in \mathbb{R}^{n+k} \, | \, Cx - Uy = d, Ax = b, y \geq 0\}. 
\end{equation}
In fact, $\text{proj}_x(Q) \subseteq P$ since $Uy \geq 0$ implies $Cx \geq d$ for any $x \in \text{proj}_x(Q)$, and $P \subseteq \text{proj}_x(Q)$ since $Cv_i - UV(:,i) = d$ implies that $(v_i,V(:,i)) \in Q$ for all $i$ and therefore each vertex $v_i$ of $P$ belongs to $\text{proj}_x(Q)$. Intuitively, this extended formulation parametrizes the space of slacks of the original polytope with the convex cone $\{ U y\, |\, y \ge 0\}$. 

 It is therefore interesting to compute bounds for the nonnegative rank in order to estimate the size of these extended formulations. Recently, Goemans \cite{G09} used this result to show that the size of LP formulations of the permutahedron (polytope whose $n!$ vertices are permutations of $[1,2,\dots,n]$) is at least $\Omega(n\log(n))$ variables plus constraints (cf.\@ Section~\ref{NR}). \\

We will see in Section~\ref{geoNN} that the nonnegative rank is closely related to a problem in computational geometry that consists in finding a polytope with minimum number of vertices nested between two given polytopes.  Therefore a better understanding of the properties of the nonnegative rank would presumably also allow to improve characterization of the solutions to this geometric problem. 

The nonnegative rank also has connections with other problems, e.g., in communication complexity theory \cite{Y91, LS09}, probability \cite{CR10}, and graph theory (cf.\@ Section~\ref{NR}).  \\

The main goal of this paper is to provide improved lower bounds on the nonnegative rank. In Section~\ref{RNR}, we introduce a new related quantity called \emph{restricted nonnegative rank}. Generalizing a recent result of Vavasis \cite{Vav} (see also \cite{CL10}), we show that computing this quantity is equivalent to a problem in polyhedral combinatorics, and fully characterize its computational complexity. 
In Section~\ref{NR}, based on the geometric interpretation of the nonnegative rank and the relationship with the restricted nonnegative rank, we derive new improved lower bounds for the nonnegative rank. Finally, in Section~\ref{illus}, we apply our results to slack matrices and linear Euclidean distance matrices. We obtain counter-examples to two conjectures of Beasly and Laffey \cite{BL09}, namely we show that the nonnegative rank of linear Euclidean distance matrices is not necessarily equal to their dimension, and that the rank of a matrix is not always greater than the nonnegative rank of its square.  
 \\

\noindent  \textbf{Notation.}   The set of real matrices of dimension $m$ by $n$ is denoted $\mathbb{R}^{m \times n}$; for $A \in \mathbb{R}^{m \times n}$, we denote the $i^{\text{th}}$ column of $A$ by $A_{:i}$ or $A(:,i)$, the $j^{\text{th}}$ row of $A$ by $A_{j:}$ or $A(j,:)$, and the entry at position $(i,j)$ by $A_{ij}$ or $A(i,j)$; for $b \in \mathbb{R}^{m \times 1} = \mathbb{R}^{m}$, we denote the $i^{\text{th}}$ entry of $b$ by $b_i$.  Notation $A(I,J)$ refers to the submatrix of $A$ with row and column indices respectively in $I$ and $J$, and $a$:$b$ is the set  $\{a,a+1,\dots,b-1,b\}$ (for $a$ and $b$ integers with $a \leq b$). 
The set $\mathbb{R}^{m \times n}$ with component-wise nonnegative entries is denoted $\mathbb{R}^{m \times n}_+$. 
The matrix $A^T$ is the transpose of $A$. 
 The rank of a matrix $A$ is denoted $\rank(A)$, its column space $\col(A)$.  The convex hull of the set of points $S$, or the convex hull of the columns of the matrix $S$ are denoted $\conv(S)$. The number of vertices of the polytope $Q$ is denoted by $\#\vertices(Q)$. The concatenation of the columns of two matrices $A \in \mathbb{R}^{m \times n}$ and $B \in \mathbb{R}^{m \times p}$ is denoted $[A \, B] \in \mathbb{R}^{m \times (n+p)}$. The sparsity pattern of a vector is the set of indices of its zero entries (it is the complement of its support).

\section{Restricted Nonnegative Rank} \label{RNR}


In this section, we analyze the following quantity
\begin{definition} 
The \emph{restricted nonnegative rank} of a nonnegative matrix $M$ is the minimum value of $k$ such that there exists $U \in \mathbb{R}^{m \times k}_+$ and $V \in \mathbb{R}^{k \times n}_+$ with $M = UV$ and ${\rank(U) = \rank(M)}$, i.e., $\col(U) = \col(M)$. It is denoted $\rank_+^*(M)$. 
\end{definition}
In particular, given a nonnegative matrix $M$, we are interested in computing its restricted nonnegative rank $\rank_+^*(M)$ and a corresponding nonnegative factorization, i.e., solve  
\begin{quote} (RNR) Given a nonnegative matrix $M \in \mathbb{R}^{m \times n}_+$, find $k = \rank_+^*(M)$ and compute $U \in \mathbb{R}^{m \times k}_+$ and $V \in \mathbb{R}^{k \times n}_+$ such that $M = UV$ and ${\rank(U) = \rank(M)} = r$. 
\end{quote}
Without the rank constraint on the matrix $U$, this problem reduces to the standard nonnegative rank problem. Motivation to study this restriction includes the following
\begin{enumerate}
\item The restricted nonnegative rank provides a new \emph{upper} bound for the nonnegative rank, since  $\rank_+(M) \leq \rank_+^*(M)$. 
\item The restricted nonnegative rank can be characterized much more easily. In particular, its geometrical interpretation (Section~\ref{equivRNR}) will lead to new improved \emph{lower} bounds for the nonnegative rank (Sections~\ref{NR} and~\ref{illus}).
\end{enumerate}

RNR is a generalization of \emph{exact nonnegative matrix factorization} (exact NMF) introduced by Vavasis \cite{Vav}. Noting $r = \rank(M)$, exact NMF asks whether $\rank_+(M) = r$ and, if the answer is positive, to compute a rank-$r$ nonnegative factorization of $M$. If $\rank_+(M) = r$ then it is clear that $\rank_+^*(M) = \rank_+(M)$ since the rank of $U$ in any rank-$r$ nonnegative factorization $(U,V)$ of $M$ must be equal to $r$. 

Vavasis studies the computational complexity of exact NMF and proves it is NP-hard by showing its equivalence with a problem in polyhedral combinatorics called intermediate simplex. This construction  requires both the dimensions of matrix $M$ and its rank $r$ to increase to obtain NP-hardness. This result also implies NP-hardness of RNR when the rank of matrix $M$ is not fixed. 
However, in the case where the rank $r$ of matrix $M$ is fixed, no complexity results are known (except in the trivial cases $r=1,2$ \cite{Tho}). The situation for RNR is quite different: we are going to show that RNR can be solved in polynomial-time when $r = 3$ and that it is NP-hard for any fixed $r \geq 4$. In particular, this result implies that exact NMF can be solved in polynomial-time for rank-three nonnegative matrices. 

In order to do so, we first show equivalence of RNR with another problem in polyhedral combinatorics, closely related to intermediate simplex (Section~\ref{equivRNR}), and then apply results from the computational geometry literature to conclude about its computational complexity for fixed rank (Section~\ref{CC}).

\subsection{Equivalence with the Nested Polytopes Problem} \label{equivRNR}

Let consider the following problem called nested polytopes problem (NPP):
\begin{quote} (NPP) Given a bounded polyhedron 
\[
P = \{ x \in \mathbb{R}^{r-1} \;|\;   0 \leq f(x) = Cx+d \},
\]
with $(C \; d) \in \mathbb{R}^{m \times r}$ of rank $r$, 
and a set $S$ of $n$ points in $P$ not contained in any hyperplane (i.e., $\conv(S)$ is full-dimensional), find the minimum number $k$ of points in $P$ whose convex hull $T$ contains $S$, i.e., $S \subseteq T \subseteq P$. 
\end{quote}
Polytope $P$ is referred to as the outer polytope, and  $\conv(S)$ as the inner polytope; note that they are given by two distinct types of representations (faces for $P$, extreme points for $\conv(S)$).

The intermediate simplex problem mentioned earlier and introduced by Vavasis \cite{Vav} is a particular case of NPP in which one asks whether $k$ is equal to $r$ (which is the minimum possible value), i.e., if there exists a simplex $T$ (defined by $r$ vertices in a $r-1$ dimensional space) contained in $P$ and containing $S$. \\

We now prove equivalence between RNR and NPP. It is a generalization of the result of  Vavasis \cite{Vav} who showed equivalence of exact NMF and intermediate simplex. 
\begin{theorem} \label{equiv}
There is a polynomial-time reduction from RNR to NPP and vice-versa.
\end{theorem}
\begin{proof}
Let us construct a reduction of RNR to NPP. First we (1) delete the zero rows and columns of $M$ and (2) normalize its columns such that $M$ becomes column stochastic (columns are nonnegative and sum to one). One can easily check that it gives a polynomially equivalent RNR instance \cite{CL08}. 
We then decompose $M$ as the product of two rank-$r$ matrices (using, e.g., reduction to row-echelon form) 
\begin{equation} \label{AB}
M = AB \iff M_{:i} = \sum_{l = 1}^r A_{:l} B_{li} \; \forall i,
\end{equation}
where $r = \rank(M)$, $A \in \mathbb{R}^{m \times r}$ and $B \in \mathbb{R}^{r \times n}$. We observe that one can assume without loss of generality that the columns of $A$ and $B$ sum to one. Indeed,  since $M$ is column stochastic, at least one column of $A$ does not sum to zero (otherwise all columns of $AB=M$ would sum to zero). One can then update $A$ and $B$ in the following way so that their columns sum to one: 
\begin{itemize}

\item For each column of $A$ which sums to zero, add a column of $A$ which does not sum to zero, and update $B$ accordingly;

\item Normalize the columns of $A$ such that they sum to one, and update $B$ accordingly;

\item Observe that since the columns of $A$ sum to one, $M$ is column stochastic and since $M = AB$, the columns of $B$ must also sum to one.
 
\end{itemize}

In order to find a solution of RNR, we have to find $U \in \mathbb{R}^{m \times k}_+$ and $V \in \mathbb{R}^{k \times n}_+$ such that $M = UV$ and $\rank(U) = r$. For the same reasons as for $A$ and $B$, $U$ and $V$ can be assumed to be column stochastic without loss of generality. Moreover, since
\[
M = UV = AB,
\]
and $\rank(M) = \rank(A) = \rank(U) = r$, the column spaces of $M$, $A$ and $U$ coincide; 
implying that the columns of $U$ must be a linear combination of the columns of $A$. The columns of $U$ must then belong to the following set
\begin{equation} \label{Q}
Q = \{ u \in \mathbb{R}^{m} \;|\; u \in \col(A), \; u \geq 0 \text{ and } \sum_{i=1}^{m} u_i = 1 \}.
\end{equation}
One can then reduce the search space to the $(r-1)$-dimensional polyhedron corresponding to the coefficients of all possible linear combinations of the columns of $A$ generating stochastic columns. Defining 
\begin{equation} \label{Cd}
C(:,i) = A(:,i)-A(:,r) \quad 1 \leq i \leq r-1, \; \text{ and } \; d = A(:,r), \nonumber
\end{equation}
and introducing affine function $f : \mathbb{R}^{r-1} \rightarrow \mathbb{R}^{m} : x \rightarrow f(x) = Cx + d$, which is injective since $C$ is full rank (because $A$ is full rank), this polyhedron can be defined as
\begin{equation} \label{Pm}
P = \{ x \in \mathbb{R}^{r-1} 
\;|\; A(:,\textrm{$1$$:$$r$$-$$1$})x + \Big(1-\sum_{i=1}^{r-1} x_i\Big) A(:,r) \geq 0 \} = \{ x \in \mathbb{R}^{r-1} 
\;|\; f(x) \geq 0 \} . 
\end{equation}
Note that $B(\textrm{$1$$:$$r$$-$$1$},j) \in P$ $\forall j$ since $M(:,j) = AB(:,j) = f(B(\textrm{$1$$:$$r$$-$$1$},j)) \geq 0$ $\forall j$. 

Let us show that $P$ is bounded: suppose $P$ is unbounded, then
\begin{eqnarray*}
\exists \, x \in P, \exists \, y \neq 0 \in \mathbb{R}^{r-1}, \forall \alpha \geq 0 
& : & x + \alpha y \in P,\\
&   & \iff C(x + \alpha y) + d = (Cx+d) + \alpha Cy \geq 0.
\end{eqnarray*}
Since $Cx + d \geq 0$, this implies that $Cy \geq 0$.
Observe that columns of $C$ sum to zero (since the columns of $A$ sum to one) so that $Cy$ sums to zero as well; moreover, $C$ is full rank and $y$ is nonzero implying that $Cy$ is nonzero and therefore that $Cy$ must contain at least one negative entry, a contradiction.

Notice that the set $Q$ can be equivalently written as 
\[
Q = \{ u \in \mathbb{R}^{m} \;|\; u = f(x), \; x \in P \}.
\]
Noting $X = [x_1 \, x_2 \, \dots x_k] \in \mathbb{R}^{r-1 \times k}$, $f(X) = [f(x_1) \, f(x_2) \, \dots f(x_k)] =  CX + [d \, d \dots d]$, we finally have
\begin{eqnarray*}
\exists U \in \mathbb{R}^{m \times k}, V \in \mathbb{R}^{k \times n}   \textrm{ column stochastic} & \text{with} &  \rank(U) = \rank(M)   \; \textrm{ and }  M = UV \\
& \iff & \\
\exists x_1, x_2, \dots x_k \in P 
\text{ and } V \in \mathbb{R}^{k \times n} \textrm{ column stochastic}  
& \textrm{s.t.} & M = f(B(\textrm{$1$$:$$r$$-$$1$},:)) = f(X) V = f(XV) \\
&  & \iff 
B(\textrm{$1$$:$$r$$-$$1$},:) = XV.
\end{eqnarray*}
The first equivalence  follows from the above derivations (i.e., $U = f(X)$ for some $x_1, x_2, \dots x_k \in P$);  $f(X)V = f(XV)$ because $V$ is column stochastic (so that $[d \, d \, \dots \, d] V = [d \, d \, \dots \, d]$), and the second equivalence is a consequence from the fact that $f$ is an injection. 

We have then reduced RNR to NPP: find the minimum number $k$ of points $x_i$ in $P$ such that $n$ given points (the columns of $B(\textrm{$1$$:$$r$$-$$1$},:)$ constructed from the columns of $M$, which define the set $S$ in the NPP instance) are contained in the convex hull of these points (since $V$ is column stochastic). 
Because all steps in the above derivation are equivalences, we have actually also defined a reduction from NPP to  RNR; to map a NPP instance to a RNR instance, we take 
\[
M(:,i) = f(s_i) = Cs_i + d \geq 0, \; s_i \in S \quad 1 \leq i \leq n, 
\]
and $\rank(M) = r$ because the $n$ points $s_i$ are not all contained in any hyperplane (they affinely span P). 
\end{proof}
It is worth noting that $M$ would be the slack matrix of $P$ if $S$ was the set of vertices of $P$ (cf.\@ Introduction). This will be useful later in Section~\ref{SM}. 


\subsection{Computational Complexity} \label{CC}

\subsubsection{Rank-Three Matrices} \label{r3m}

Using Theorem~\ref{equiv}, RNR  of a rank-three matrix can be reduced to a two-dimensional nested polytopes problem\footnote{See also Appendix~\ref{2Drep} where a \matlab code is provided.}.
 Therefore, one has to find a convex polygon $T$ with minimum number of vertices nested in between two given convex polygons $S \subset P$. 
This problem has been studied by Aggarwal et al.\@ in \cite{ABOS89}, who proposed an algorithm running in $O(p \log(k))$ operations\footnote{Wang generalized the result for non-convex polygons \cite{W91}. Bhadury and Chandrasekaran propose an algorithm to compute all possible solutions \cite{BC96}.}, where $p$ is the total number of vertices of the given polygons $S$ and $P$, and $k$ is the number of vertices of the minimal nested polygon $T$. 
If $M$ is a $m$-by-$n$ matrix then  $p \leq m+n$ since $S$ has $n$ vertices, and the polygon $P$ is defined by $m$ inequalities so that it has at most $m$ vertices. Moreover $k = \rank_+^*(M) \leq \min(m,n)$ follows from the trivial solutions $T = S$ and $T = P$. 
Finally, we conclude that one can compute the restricted nonnegative rank of a rank-three $m$-by-$n$ matrix in $O\big((m+n) \log(\min(m,n))\big)$ operations. 
\begin{theorem} \label{exactNMFcompl}
For $\rank(M) \leq 3$,  RNR can be solved in polynomial-time. 
\end{theorem}
\begin{proof}
Cases $r=1,2$ are trivial since any rank-1 (resp.\@ 2) nonnegative matrix can always be expressed as the sum of 1 (resp.\@ 2) nonnegative factors \cite{Tho}. 

Case $r=3$ follows from Theorem~\ref{equiv} and the polynomial-time algorithm of Aggarwal et al.\@ \cite{ABOS89}. 
\end{proof}

For the sake of completeness, we sketch the main ideas of the algorithm of Aggarwal et al. 
They first make the following observations: (1) any vertex of a solution $T$ can be assumed to belong to the boundary of the polygon $P$ (otherwise it can be projected back on $P$ in order to generate a new solution containing the previous one), (2) any segment whose ends are on the boundary of $P$ and tangent to $S$ (i.e., $S$ is on one side of the segment, and the segment touches $S$) defines a polygon with the boundary of $P$ which must contain a vertex of any feasible solution $T$ (otherwise the tangent point on $S$ could not be contained in $T$), see, e.g., set $Q$ on Figure~\ref{illusAgg} delimited by the segment $[p_1,p_2]$ and the boundary of $P$, and such that $T \cap Q \neq \emptyset$ for any feasible solution $T$. 
\begin{figure*}[ht!]
\begin{center}
\includegraphics[width=6cm]{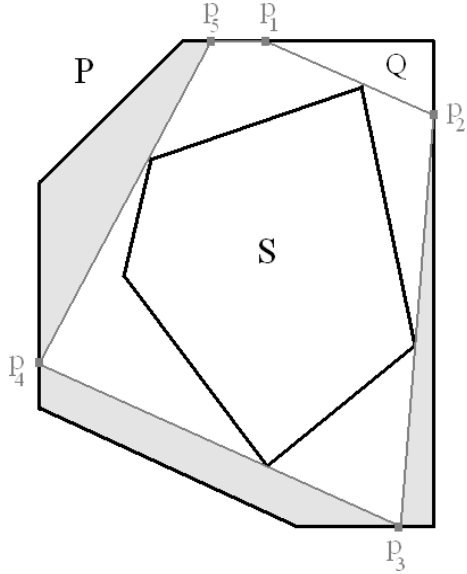}
\caption{Illustration of the algorithm of Aggarwal et al.\@ \cite{ABOS89}.}
\label{illusAgg}
\end{center}
\end{figure*}

Starting from any point $p_1$ of the boundary of $P$, one can trace the tangent to $S$ and hence  obtain the next intersection $p_2$ with $P$. Point $p_2$ is chosen as the next vertex of a solution $T$, and the same procedure is applied (say $k$ times) until the algorithm  can reach the initial point without going through $S$, see Figure~\ref{illusAgg}. This generates a feasible solution $T(p_1) = \conv(\{p_1,p_2, \dots, p_k\})$.  
Because of (1) and (2), this solution has at most one vertex more than an optimal one, i.e., $k \leq \rank_+^*(M) +1$ (since $T$ determines with the boundary of $P$ $k-1$ disjoint polygons tangent to $S$). Moreover, because of (1) and (2), there must exist a vertex of an optimal solution on the boundary of $P$ between $p_1$ and $p_2$. 

The point $p_1$ is then replaced by the so called `contact change points' located on this part of the boundary of $P$ while the corresponding solution $T(p_1)$ is updated using the procedure described above. The contact change points are: (a) the vertices of $P$ between $p_1$ and $p_2$, and (b) the points for which one tangent point of $T(p_1)$ on $S$ is changed when $p_1$ is replaced by them. This (finite) set of points provides a list of candidates where the number of vertices of the solution $T$ could potentially be reduced (i.e., where $p_1$ and $p_k$ could coincide) by replacing $p_1$ by one of these points.  
It is then possible to check whether the current solution can be improved or not, and  guarantee global optimality. In the example of Figure~\ref{illusAgg}, moving $p_1$ on the (only) vertex of $P$ between $p_1$ and $p_2$ generates an optimal solution of this RNR instance (since it reduces the original solution from 5 to 4 vertices).  


\subsubsection{Higher Rank Matrices}

For a rank-four matrix, RNR reduces to a three-dimensional problem of finding a polytope $T$ with the minimum number of vertices nested between two other polytopes $S \subseteq P$. This problem has been studied by Das et al.\@ \cite{DJ90, gau} and  has been shown to be NP-hard when minimizing the number of faces of $T$ (the reduction is from \emph{planar-3SAT}). From this result, one can deduce using a duality argument\footnote{Taking the polar of the three nested polytopes exchanges the roles of the inner and outer polytopes, and transforms face descriptions into vertex descriptions, so that the description of the inner and outer polytopes is unchanged but the intermediate polytope is now described by its vertices.} that minimizing the number of vertices of $T$ is NP-hard as well \cite{DG97,C90}. 
\begin{theorem} \label{complRNR}
For $\rank(M) \geq 4$,  RNR is NP-hard.  
\end{theorem}
\begin{proof} 
This is a consequence of Theorem~\ref{equiv} and the NP-hardness results of Das et al.\@ \cite{DJ90, gau, DG97}.
\end{proof}


Note however that several approximation algorithms 
have been proposed in the literature. For example, Mitchell and Suri \cite{MS92} approximate  $\rank_{+}^*(M)$ in case $\rank(M) = 4$ within a $O(\log(p))$ factor, where $p$ is the total number of vertices of the given polygons $S$ and $P$. 
Clarkson \cite{C90} proposes a randomized algorithm finding a polytope $T$ with at most $r_+^* O(5 d \ln(r_+^*))$ vertices and running in $O({r_+^*}^2p^{1+\delta})$ expected time (with $r_+^* = \rank_{+}^*(M)$, $d = \rank(M) - 1$ and $\delta$ is any fixed value $> 0$). 

\subsection{Some Properties}

In this section, we derive some useful properties of the restricted nonnegative rank. 
\begin{example} \label{ex1}
Construct $M$ using the following NPP instance:  $P$ is the three dimensional cube
$P = \{ x \in \mathbb{R}^3 \; | \; 0 \leq x_i \leq 1, \; 1 \leq i \leq 3\}$,
with 6 faces and $S$ is the set of its 8 vertices  
$S = \{ x \in \mathbb{R}^3 \; | \; x_i \in \{ 0, 1\}, \; 1 \leq i \leq 3\}$. 
By construction, the convex hull of $S$ is equal to $P$ and the unique and optimal solution to this NPP instance is $T = P = \conv(S)$ with 8 vertices. By Theorem~\ref{equiv}, the corresponding matrix $M$ of the RNR instance 
\begin{equation} \label{biclique}
M  
 =  
\left( \begin{array}{cccccccc}
     1  &   0  &   1  &   1  &   0  &   0  &   1   &  0 \\
     0  &   1  &   0  &   0  &   1  &   1  &   0   &  1 \\
     1  &   1  &   0  &   1  &   0  &   1  &   0   &  0 \\
     0  &   0  &   1  &   0  &   1  &   0  &   1   &  1 \\ 
     1  &   1  &   1  &   0  &   1  &   0  &   0   &  0 \\
     0  &   0  &   0  &   1  &   0  &   1  &   1   &  1 \\
\end{array} \right) 
\end{equation}
has restricted nonnegative rank equal to 8 (note that its rank is 4 and its nonnegative rank is 6, see Section~\ref{NR}). 
\end{example}

It is well-known that for a matrix $M \in \mathbb{R}^{m \times n}_+$, we have $\rank_+(M) \le \min(m,n)$; surprisingly, this does not hold for the restricted nonnegative rank.
\begin{lemma} \label{weirdo}
For $M \in \mathbb{R}^{m \times n}_+$, 
\[
\rank_+^*(M) \leq n \quad \text{ but } \quad \rank_+^*(M) \nleq m.
\]
\end{lemma}
\begin{proof}
The first inequality is trivial since $M = M I$ ($I$ being the identity matrix). Example~\ref{ex1} provides an example when $\rank_+^*(M) = 8$ for a 6-by-8 matrix $M$. 
\end{proof}

Lemma~\ref{weirdo} implies that in general $\rank_+^*(M) \neq \rank_+^*(M^T)$, unlike the rank and nonnegative rank \cite{CR93}. Note however that when $\rank(M) \leq 3$, we have 
\[
\rank_+^*(M) \leq \min(m,n),
\] 
because the number of vertices of the outer polygon $P$ in the NPP instance is smaller or equal to its number of facets $m$ in the two-dimensional case or lower, and that the solution $T=P$ is always feasible. \\

\begin{lemma} Let $A \in \mathbb{R}^{m \times n}_+$ and $B \in \mathbb{R}^{m \times r}_+$, then
\[
\rank_+^*([A \, B]) \leq \rank_+^*(A) + \rank_+^*(B).
\]
\end{lemma}
\begin{proof}
Let $(U_a,V_a)$ and $(U_b,V_b)$ be solutions of RNR for $A$ and $B$ respectively, then 
\[ 
[A \, B] = [U_a \, U_b] \left[ \begin{array}{cc}
     V_a   & 0  \\
     0  &   V_b  \end{array} \right],
\]
and $\rank([U_a \, U_b]) = \rank([A \, B])$ since $\col(U_a) = \col(A)$ and $\col(U_b) = \col(B)$ by definition. 
\end{proof}



\begin{lemma} \label{rMrU} 
Let $M \in \mathbb{R}^{m \times n}_+$ with $\rank(M) = r$ and $\rank_+(M) = r_+$, $U \in \mathbb{R}^{m \times r_+}_+$ and $V \in \mathbb{R}^{r_+ \times n}_+$ with $M = UV$. Then  
\[
r_+ < \rank_+^*(M) \quad \Rightarrow \quad r < \rank(U) \leq r_+ \; \text{ and } \;  r \leq \rank(V) < r_+.
\]
Moreover, if $M$ is symmetric,
\[
r_+ < \rank_+^*(M) \quad \Rightarrow \quad r < \rank(U) < r_+ \; \text{ and } \;  r < \rank(V) < r_+.
\]
\end{lemma}
\begin{proof} 
Clearly,
\[
r \leq \rank(U) \leq r_+ \quad \text{ and } \quad r \leq \rank(V) \leq r_+.
\]
If $\rank(U) = r$, we would have $\rank_+^*(M) =  r_+$ which is a contradiction, and $\rank(V) =  r_+$ would imply that $V$ has a right pseudo-inverse $V^{\dag}$, so that we could write $U = MV^{\dag}$ and then $r \leq \rank(U) \leq \min(r,\rank(V^{\dag})) \leq  r$, a contradiction for the same reason. 

In case $M$ is symmetric, to show that $\rank(U) <  r_+$ and $\rank(V) > r$, we use symmetry and observe that 
$UV = M = M^T = V^T U^T$.
\end{proof}

\begin{corollary} \label{nr5}
Given a nonnegative matrix $M$, 
\[
\rank_+^*(M) \leq \rank(M)+1 \quad \Rightarrow \quad \rank_+(M) = \rank_+^*(M).
\]
If $M$ is symmetric,
\[
\rank_+^*(M) \leq \rank(M)+2 \quad \Rightarrow \quad \rank_+(M) = \rank_+^*(M).
\]
\end{corollary}
\begin{proof}
Let $r = \rank(M)$, $r_+ = \rank_+(M)$, and $U \in \mathbb{R}^{m \times r_+}_+$ and $V \in \mathbb{R}^{r_+ \times n}_+$ such that $M = UV$. If $r_+ < \rank_+^*(M)$, by Lemma~\ref{rMrU}, we  have
\[
r < \rank(U) \leq r_+ < \rank_+^*(M),
\]
which is a contradiction if $\rank_+^*(M) \leq r+1$. If $M$ is symmetric, we have $\rank(U) < r_+$ and the above equation is a contradiction if $\rank_+^*(M) \leq r+2$.
\end{proof}

For example, this implies that to find a symmetric rank-three nonnegative matrix with $\rank_+(M) < \rank_+^*(M)$, we need $\rank_+^*(M) > \rank(M) + 2 = 5$ and therefore have to consider matrices of size at least 6-by-6 with $\rank_+^*(M) = 6$. 
\begin{example} \label{ex2}
Let us consider the following matrix $M$ and the rank-5 nonnegative factorization $(U,V)$, 
\begin{equation} 
M   =   
\left( \begin{array}{cccccc}
     0   &  1  &   4   & 9   & 16 &   25\\
     1  &   0  &   1 &    4   &  9  &  16\\
     4   &  1  &   0 &    1   &  4  &   9\\
     9  &   4  &   1  &   0   &  1  &   4\\
    16   &  9   &  4  &   1   &  0  &   1\\
    25  &  16   &  9  &   4    & 1  &   0
\end{array} \right) = UV, 
U =    \left( \begin{array}{ccccc}
		  5     &  0  &   4  &   0    &  1 \\
      3     &  0  &   1  &  1    &  0\\
      1     &  0   &  0   & 4     & 1\\
      0     &  1   &  0  &  4    &  1\\
      0      &  3  &   1  & 1     &  0\\
      0     &  5  &   4  &  0    &  1
      \end{array} \right), 
V =   \left( \begin{array}{cccccc}
        0   &  0   &  0    & 1    & 3   &  5\\
         5  &   3  &   1 &    0  &   0&     0\\
         0  &   0  &   1  &   1   &  0  &   0\\
          1 &    0 &    0 &    0  &   0  &   1\\
         0   &  1  &   0  &   0   &  1   &  0    \end{array} \right). \nonumber
\end{equation}
One can check that $\rank(M) = 3$, and, using the algorithm of Aggarwal et al.\@ \cite{ABOS89}, 
the restricted nonnegative rank can be computed\footnote{The problem is actually trivial because each vertex of the inner polygon $S$ is located on a different edge of the polygon $P$, so that they define with the boundary of $P$ $6$ disjoint polygons tangent to $S$. This implies that $\rank_+^*(M) = 6$, cf.\@ Section~\ref{r3m}. This matrix is actually a linear Euclidean distance matrix which will be analyzed later in  Section~\ref{euclid}.} and is equal to 6.  Using the above decomposition, it is clear that $\rank_+(M) \leq 5 < \rank_+^*(M) = 6$. 
 By Lemma~\ref{rMrU}, for any $\rank_+(M)$-nonnegative factorization $(U,V)$ of $M$, we then must have  $3 < \rank(U) = 4 < \rank_+(M)$ implying  that $\rank_+(M) = 5$.   
\end{example} 

As we have already seen with Lemma~\ref{weirdo}, the restricted nonnegative rank does not share all the nice properties of the rank and the nonnegative rank functions \cite{CR93}. The next two lemmas exploit Example~\ref{ex2} further to show different behavior between nonnegative rank and restricted nonnegative rank. 
\begin{lemma} 
Let $A \in \mathbb{R}^{m \times n}_+$ and $B \in \mathbb{R}^{m \times r}_+$, then 
\[
\rank_+^*(A+B) \nleq \rank_+^*(A) + \rank_+^*(B).
\]
\end{lemma}
\begin{proof}
Take $M$, $U$ and $V$ from Example~\ref{ex2} and construct $A = U(:,1$:$3)V(1$:$3,:)$ with $\rank_+^*(A) = 3$ (since $\rank(A) = 3$), $B = U(:,4$:$5)V(4$:$5,:)$ with $\rank_+^*(B) = 2$ (trivial) and $\rank_+^*(A+B) = 6$ since $A+B = M$. 
\end{proof}

\begin{lemma} 
Let $A \in \mathbb{R}^{m \times n}_+$ and $B \in \mathbb{R}^{m \times r}_+$, then
\[
\rank_+^*([A \, B]) \ngeq \rank_+^*(A).
\]
where $[A \, B] \in \mathbb{R}^{m \times (n+r)}_+$ denotes the concatenation of the columns of $A$ and $B$. 
\end{lemma}
\begin{proof}
Let us take $M$, $U$ and $V$ from Example~\ref{ex2}, and construct $A = M$ and $B = U(:,1)$ 
with $\rank_+^*([A \, B]) \leq 5$ since $\rank([A \, B]) = 4$ (this can be checked easily) and $[A \, B] = U [V \, e_1]$ with $\rank(U) = 4$ (where $e_i$ denotes the $i^{\text{th}}$ column of the identity matrix of appropriate dimension). 
\end{proof}

\begin{lemma} Let $B \in \mathbb{R}^{m \times r}_+$ and $C \in \mathbb{R}^{r \times n}_+$, then
\[
\rank_+^*(BC) \nleq \min(\rank_+^*(B),\rank_+^*(C)).
\]
\end{lemma}
\begin{proof}
See Example~\ref{ex2} in which $\rank_+^*(M) = 6$ and $\rank_+^*(U) \leq 5$ by Lemma~\ref{weirdo}. 
\end{proof}

\section{Lower Bounds for the Nonnegative Rank} \label{NR}

In this section, we provide new lower bounds for the nonnegative rank based on the restricted nonnegative rank. Recall that the restricted nonnegative rank already provides an upper bound for the nonnegative rank since for a $m$-by-$n$ nonnegative matrix $M$, 
\begin{equation}
0 \leq \rank(M) \leq \rank_+(M) \leq \rank_{+}^*(M) \leq n. 
\end{equation} 
Notice that this bound can only be computed efficiently in the case $\rank(M) = 3$ (see Theorems~\ref{exactNMFcompl} and \ref{complRNR}). 

As mentioned in the introduction, it might also be interesting to compute lower bounds on the nonnegative rank. 
Some work has already been done in this direction, including the following
\begin{enumerate}
%
%

\item Let $M \in \mathbb{R}^{m \times n}_+$ be any weighted biadjacency matrix of a bipartite graph $G = (V_1 \cup V_2, E \subset V_1 \times V_2)$ with $M(i,j) > 0 \iff (V_1(i),V_2(j)) \in E$. 
A biclique of $G$ is a complete bipartite subgraph (it corresponds to a positive rectangular submatrix of $M$). One can easily check that each rank-one factor $(U_{:k},V_{k:})$ of any rank-$k$ nonnegative factorization $(U,V)$ of $M$ can be interpreted as a biclique of $M$ (i.e., as a positive rectangular submatrix) since $M = \sum_{i=1}^k U_{:i}V_{i:}$. 
Moreover, these bicliques $(U_{:k},V_{k:})$ must cover $G$ completely since $M = UV$. The minimum number of bicliques needed to cover $G$ is then a lower bound for the nonnegative rank. It is called the biclique partition number and denoted $b(G)$, see \cite{W06} and references therein. Its computation is NP-complete  \cite{O77} and is directly related to the minimum biclique cover problem (MBC).\\ 
Consider for example the matrix $M$ from Example~\ref{ex1}. The largest biclique of the graph $G$ generated by $M$ has 4 edges\footnote{This can be computed explicitly, e.g., with a brute force approach. Note however that finding the biclique with the maximum number of edges is a combinatorial NP-hard optimization problem \cite{Peet}. It is closely related to a variant of the approximate nonnegative factorization problem \cite{GG08}.}. Since $G$ has 24 edges, we have $b(G) \geq \frac{24}{4}=6$ and therefore $6 \leq \rank_+(M) \leq \min(m,n)=6$.

A crown graph $G$ is a bipartite graph with $|V_1|= |V_2| = n$ and $E = \{ (V_1(i),V_2(j)) \, | \, i \neq j\}$ (it can be viewed as a biclique where the horizontal edges have been removed). 
de Caen, Gregory and Pullman \cite{CGP81} showed that
\[
b(G) = \min_k \Big\{ k \;|\; n \leq \binom{k}{\lfloor k/2 \rfloor} \Big\} = \mathcal{O}(\log n) .
\]

Beasley and Laffey \cite{BL09}  studied linear Euclidean distance matrices defined as $M(i,j) = (a_i-a_j)^2$ for $1 \leq i, j \leq n$, $a_i \in \mathbb{R}$, $a_i \neq a_j$ $i \neq j$. They proved that such matrices have rank three and that
\begin{equation} \label{beas}
\min_k \Big\{ k \;|\; n \leq \binom{k}{\lfloor k/2 \rfloor} \Big\} \leq r_+
\quad \text{ which means } \quad
n \leq 
\binom{r_+}{\lfloor r_+/2 \rfloor}, 
\end{equation}
where $r_+ = \rank_+(M)$. In fact, such matrices are biadjacency matrices of crown graphs (only the  diagonal entries are equal to zero). 

\item Goemans makes \cite{G09} the following observation: the product $UV$ of two nonnegative matrices $U \in \mathbb{R}^{m \times k}_+$ and $V \in \mathbb{R}^{k \times n}_+$ generates a matrix $M$ with at most $2^k$ columns (resp.\@ $2^k$ rows) with different sparsity patterns. In fact, the columns (resp.\@ rows) of $M$ are additive linear combinations of the $k$ columns of $U$ (resp.\@ rows of $V$) and therefore no more than $2^k$ sparsity patterns can be generated from these columns (resp.\@ rows). Therefore, 
letting $s_p$ be the maximum between the number of columns and rows of $M \in \mathbb{R}^{m \times n}_+$ having a different sparsity pattern, we have  
\[
\rank_+(M) \geq \log_2(s_p).
\]
In particular, if all the columns and rows of $M$ have a different sparsity pattern, then
\begin{equation} \label{goem}
\rank_+(M) \geq \log_2(\max(m,n)).
\end{equation}
Goemans then uses this result to show that any extended formulation of the permutahedron in dimension $n$ must have $\Omega(n\log(n))$ variables and constraints. In fact, 
\begin{itemize}
\item The minimal size $s$ is of the order of the nonnnegative rank of its slack matrix plus $n$ (cf.\@ Introduction).
\item The slack matrix has $n!$ columns (corresponding to each vertex of the polytope) with different sparsity patterns (cf.\@ Equation~\eqref{SMdef}).
\end{itemize}
This implies that 
\[
s = \Theta(\rank_+(S_M) + n) \geq \Theta(\log(n!)) = \Theta(n\log(n)).
\]
\end{enumerate}

In this section, we provide some theoretical results linking the restricted nonnegative rank with the nonnegative rank, which allow us to improve and generalize the above results in Section~\ref{illus} for both slack and linear Euclidean distance matrices.

\subsection{Geometric Interpretation of a Nonnegative Factorization as a Nested Polytopes Problem} \label{geoNN}

In the following, we lay the groundwork for the main results of this paper, introducing essential notations and observations that will be extensively used in this section. We rely on the geometric interpretation of the nonnegative rank, see also \cite{DS03, CL08, Vav} where similar results are presented.  
The main observation is that any rank-$k$ nonnegative factorization $(U,V)$ of a nonnegative matrix $M$ can be interpreted as the solution with $k$ vertices of a nested polytopes problem in which the inner polytope has dimension $\rank(M)-1$ and the outer polytope has dimension $\rank(U)-1$. \\

Without loss of generality, let $M \in \mathbb{R}^{m \times n}_+$, $U \in \mathbb{R}^{m \times k}_+$ and $V \in \mathbb{R}^{k \times n}_+$ be column stochastic with $M = UV$ (cf.\@ proof of Theorem~\ref{equiv}, the columns of $M$ are convex combination of the columns of $U$).  
If the column space of $U$ does not coincide with the column space of $M$, i.e., $r_u = \rank(U) > \rank(M) = r$, it means that the columns of $U$ belong to a higher dimensional affine subspace containing the columns of $M$ (otherwise, see Theorem~\ref{equiv}). 

Let factorize $U = AB$ where $A \in \mathbb{R}^{m \times r_u}$ and $B \in \mathbb{R}^{r_u \times r_+}$ are full rank and their columns sum to one. As in Theorem~\ref{equiv}, we can construct the polytope of the coefficients of the linear combinations of the columns of $A$ that generate stochastic vectors. It is defined as
\[
P_u =  \{ x \in \mathbb{R}^{r_u-1} 
\;|\; f_u(x) = A(:,\textrm{$1$$:$$r_u$$-$$1$})x + \Big(1-\sum_{i=1}^{r_u-1} x_i\Big) A(:,r_u) \geq 0\}.
\]

Since $\col(M) \subseteq \col(U)$, there exists $B' \in \mathbb{R}^{r_u \times n}$ whose columns must sum to one such that $M = AB'$. Since $\rank(M) = r$ and $A$ is full rank, we must have $\rank(B') = r$. By construction, the columns of $B_u = B(\textrm{$1$$:$$r_u$$-$$1$},:)$ (corresponding to the columns of $U$) and $B_m = B'(\textrm{$1$$:$$r_u$$-$$1$},:)$ (corresponding to the columns of $M$) belong to $P_u$. 
Note that since $\rank(B') = r$, the columns of $B_m$ live in a lower ($r-1$)-dimensional polytope
\[
P_m =  \{ x \in \mathbb{R}^{r_u-1} 
\;|\; f_u(x) \geq 0,  f_u(x) \in \col(M)   \} \quad \subseteq \quad P_u.
\]
Polytope $P_m$ contains the points in $P_u$ generating vectors in the column space of $M$. 

Moreover
\[
M = A B' = U V = AB V,
\]
implying that (since $A$ is full rank) 
\[
B' = B V \quad \text{ and } \quad  B_m = B_u V. 
\]
Finally, the columns of $B_m$ are contained in the convex hull of the columns of $B_u$, inside $P_u$, i.e., 
\[
\conv(B_m) \; \subseteq \; \conv(B_u) \; \subseteq \; P_u.
\] 
Defining the polytope $T$ as the convex hull of the columns of $B_u$, and the set of points $S$ as the columns of $B_m$, we can then interpret the nonnegative factorization $(U,V)$ of $M$ as follows. The $(r_u-1)$-dimensional polytope $T$ with $k$ vertices (corresponding to the columns of $U$) is nested between a inner $(r-1)$-dimensional polytope $\conv(S)$ (where each point in $S$ corresponds to a column of $M$) and a outer $(r_u-1)$-dimensional polytope $P_u$.

Let us use the matrix $M$ and its nonnegative factorization $(U,V)$ of Example~\ref{ex2} as an illustration:  $\rank(M) = 3$ so that $P_m$ is a two-dimensional polytope and contains the set of points $S$, while $\rank(U) = 4$ and defines a three-dimensional polytope $T$ containing $S$, see Figure~\ref{euc3}. 
\begin{figure*}[ht!] 
\begin{center}
\includegraphics[width=13cm]{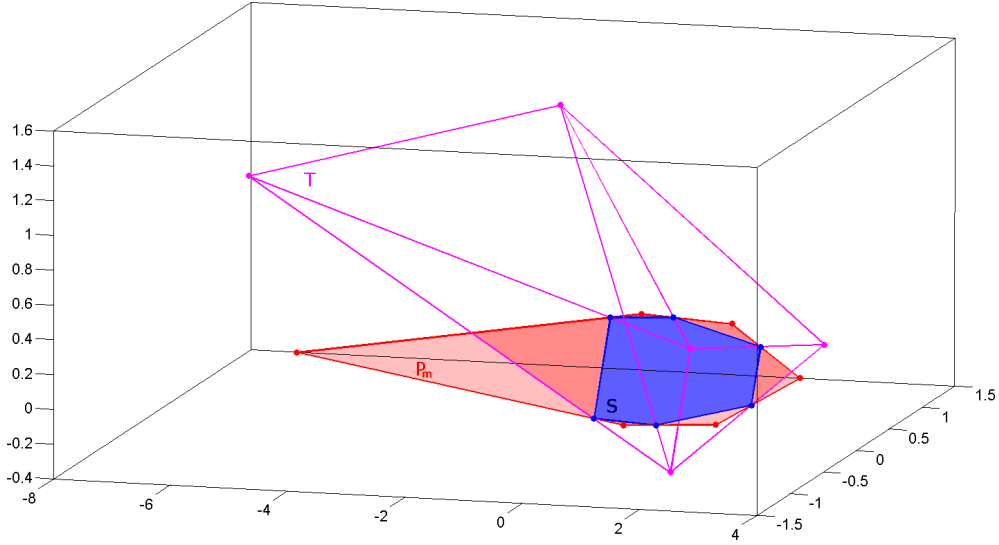}
\caption{Illustration of the solution from  Example~\ref{ex2} as a nested polytopes problem, with $\rank(M) = 3 < \rank(U) = 4 < \rank_+(M) = 5 < \rank_+^*(M) = 6 = n$. See Appendix~\ref{3Drep} for the code used to perform the reduction.}
\label{euc3}
\end{center}
\end{figure*}



\subsection{Upper Bound for the Restricted Nonnegative Rank}

From the geometric interpretation introduced in the previous paragraph, we can now give the main result of this section. 
The idea is the following: using notations of Section~\ref{geoNN}, we know that (1) the polytope $T$ (whose vertices correspond to the columns of $U$) contains the (lower dimensional) set of points  $S$ (corresponding to the columns of $M$), and (2) $S$ is contained in $P_m$ (which corresponds to the set of stochastic vectors in the column space of $M$). 
Therefore, the intersection between $T$ and $P_m$ must also contain $S$, i.e., the intersection $T \cap P_m$ defines a polytope which  (1) is contained in the column space of $M$, and (2) contains $S$.  Hence its vertices provide a feasible solution to the RNR problem, and an upper bound for the restricted nonnegative rank can then be computed. 

In other words, any nonnegative factorization $(U,V)$ of a nonnegative matrix $M$ can be used to construct a feasible solution to the restricted nonnegative rank problem. One has simply to compute the intersection of the polytope generated by the columns of $U$ with the column space of $M$ (which can obviously increase the number of vertices). 

\begin{theorem} \label{bornernr}  Using notations of Section~\ref{geoNN},  we have 
\begin{equation} \label{lbnn}
\rank_+^*(M) \leq \#\vertices(T \cap P_m). 
\end{equation}
\end{theorem}
\begin{proof} 
Let $x_1, x_2, \dots, x_v$ be the $v$ vertices of $T \cap P_m$ and note $X = [x_1 \, x_2 \, \dots x_v]$ which has rank at most $r$ (since it is contained in the $(r-1)$-dimensional polyhedron $P_m$). By construction, 
\begin{equation} \label{Bmtp}
B_m(:,j) \; \in \; T \cap P_m = \conv(X) \quad 1 \leq j \leq n. \nonumber
\end{equation} 
Therefore, there must exist a matrix $V^* \in \mathbb{R}^{v \times n}$ column stochastic such that
\[
B_m = X V^*,
\]
implying that
\[
M = f_u(B_m) 
= f_u(X V^*) 
= f_u(X) V^* = U^*V^*,
\]
where $U^* = f_u(X) \in \mathbb{R}^{m \times v}$ is nonnegative since $x_i \in P_m \subset P_u \, \forall i$, and $U^*$ has rank $r$ since $M = U^*V^*$ implies that its rank is at least $r$ and $U^* = f_u(X)$ that it is at most $r$.   
The pair $(U^*,V^*)$ is then a feasible solution of the corresponding RNR problem for $M$ and therefore  $\rank_+^*(M) \leq v = \#\vertices(T \cap P_m)$.  
\end{proof}

%
%
%



\subsection{Lower Bound for the Nonnegative Rank based on the Restricted Nonnegative Rank}

We can now obtain a lower bound for the nonnegative rank based on the restricted nonnegative rank. Indeed, if we consider an upper bound on the quantity $\#\vertices(T \cap P_m)$ that increases with the nonnegative rank (i.e., the number of vertices of $T$), we can reinterpret Theorem~\ref{bornernr} as providing a lower bound on the nonnegative rank. For that purpose, define the quantity $\faces(n,d,k)$ to be the maximal number of $k$-faces of a polytope with $n$ vertices in dimension $d$.
\begin{theorem} \label{Thbound1}
The restricted nonnegative rank of a nonnegative matrix $M$ with $r = \rank(M)$ and $r_+ = \rank_+(M)$ can be bounded above by 
\begin{equation}
\rank_+^*(M)  \leq  \max_{r \leq r_u \leq r_+} \faces(r_+,r_u-1,r_u-r).  \label{bound1}
\end{equation}
 
\end{theorem}
\begin{proof}
Let $(U,V)$ be a rank-$r_+$ nonnegative factorization of $M$ with $\rank(U) = r_u$. 
Using notations of Section~\ref{geoNN} and the result of Theorem~\ref{bornernr}, $\rank_+^*(M)$ is bounded above by the the number of vertices of $T \cap P_m$.  Defining $Q_m =  \{ x \in \mathbb{R}^{r_u-1} \,|\, f_u(x) \in \col(M)\}$, we have $P_m =  Q_m \cap P_u$ and since $T \subset P_u$,
\[
P_m \cap T = Q_m \cap P_u \cap T = Q_m \cap T.
\]
Since $Q_m$ is $(r-1)$-dimensional, the number of vertices of $T \cap Q_m$ is bounded above by the number of $(r_u-r)$-faces of $T$ (in a $(r_u-1)$-dimensional space, $(r_u-r)$-faces are   defined by $r-1$ equalities), we then have
\[
\rank_+^*(M) \leq \#\vertices(T \cap P_m) = \#\vertices(T \cap Q_m) \leq \faces(r_+,r_u-1,r_u-r).
\] 
Notice that for $r_u = r$, $\faces(r_+,r-1,0) = r_+$ which gives $r_+ = \rank_+^*(M)$ as expected. 
Finally, taking the maximum over all possible values of $r \leq r_u \leq r_+$ gives the above bound \eqref{bound1}. 
\end{proof}
We introduce for easier reference a function $\phi$ corresponding to the upper bound in Theorem~\ref{Thbound1}, i.e.,
\[
\phi(r,r_+) =  \max_{r \leq r_u \leq r_+} \faces(r_+,r_u-1,r_u-r).
\]
Clearly, when $r$ is fixed, $\phi$ is an increasing function of its second argument $r_+$, since $\faces(n,d,k)$ increases with $n$. Therefore inequality $\rank_+^*(M) \leq \phi(r,r_+)$ from Theorem~\ref{Thbound1} implicitly provides a lower bound on the nonnegative rank $r_+$ that depends on both rank $r$ and restricted nonnegative rank $\rank_+^*(M)$.

Explicit values for function $\phi$ can be computed using a tight bound for $\faces(n,d,k)$ attained by cyclic polytopes \cite[p.257, Corollary 8.28]{Z95}
\[
\faces(n,d,k-1) = \sum_{i = 0}^{\frac{d}{2}}{}^* \bigg(  
\binom{d-i}{k-i} + \binom{i}{k-d+i} \bigg) \binom{n-d-1+i}{i},  
\]
where $\sum{}^*$ denotes a sum where only half of the last term is taken for $i = \frac{d}{2}$ if $d$ is even, and the whole last term is taken for $i = \lfloor \frac{d}{2} \rfloor = \frac{d-1}{2}$ if $d$ is odd. 
Alternatively, simpler versions of the bound can be worked out in the following way:
\begin{theorem}\label{allbounds}The upper bound $\phi(r,r_+)$ on the restricted nonnegative rank of a nonnegative matrix $M$ with $r = \rank(M)$ and $r_+ = \rank_+(M)$ satisfies 
\begin{eqnarray*}
\phi(r,r_+) &=& \max_{r \leq r_u \leq r_+} \faces(r_+,r_u-1,r_u-r) \\ 
&\le& \max_{r \leq r_u \leq r_+} \binom{r_+}{r_u-r+1} \le \binom{r_+}{\lfloor r_+/2 \rfloor} \le  2^{r_+} \sqrt{\frac{2}{\pi r_+}}\le 2^{r_+} \;.
\end{eqnarray*} 
\end{theorem}
\begin{proof}
The first inequality follows from the fact that $\faces(n,d,k-1) \leq \binom{n}{k}$, since any set of $k$ distinct vertices defines at most one $k-1$-face. The second follows from the maximality of central binomial coefficients. The third is a standard upper bound on central binomial coefficients, and the fourth is an even cruder upper bound.\end{proof}
We will see in Section~4 that some of these weaker bounds correspond to existing results from the literature.

When matrix $M$ is symmetric, the bound can be slightly strengthened, leading to a different function $\phi'$:
\begin{corollary} \label{corolrp}
Given a symmetric matrix $M$ with $r_+ = \rank_+(M)$, $r = \rank(M)$ and $r_+ \geq r+1$, we have 
\[ 
\rank_+^*(M) \leq \max_{r \leq r_u \leq r_+-1} \faces(r_+,r_u-1,r_u-r) = \phi'(r,r_+)  \leq \phi(r,r_+). 
\]
\end{corollary}
\begin{proof}
We have seen in Lemma~\ref{rMrU} that for symmetric matrices $r_u = r_+$ implies  $\rank_+^*(M) = r_+$. Therefore, in case $r_+ \geq r+1$, one can strengthen the result of Theorem~\ref{Thbound1} and only consider the range $r \leq r_u \leq r_+-1$. 
\end{proof}

\subsubsection{Improvements in the rank-three case}

It is possible to improve the above bound by finding better upper bounds for $\#\vertices(T \cap P_m)$ in Equation~\eqref{lbnn}. For example, since two-dimensional polytopes (i.e., polygons) have the same number of vertices (0-faces) and edges (1-faces), we have for $\rank(M) = 3$ that
\[
\#\vertices(T \cap P_m) \quad = \quad \#\text{edges}(T \cap P_m).
\] 
Using the same argument as in Theorem~\ref{Thbound1}, the number of edges of $T \cap P_m$ is bounded above by the number of $(r_u-r+1)$-faces of $T$ (defined by $r-2$ equalities) 
leading to 
\begin{corollary} \label{cor3}
The restricted nonnegative rank of a rank-three nonnegative matrix $M$ with $r_+ = \rank_+(M)$ can be bounded above with
\begin{equation} \label{boundr3}
\rank_+^*(M)  \leq   \max_{3 \leq r_u \leq r_+} \; \min_{i=0,1} \; \faces(r_+,r_u-1,r_u-3+i) \le \phi(3,r_+).
\end{equation}
\end{corollary} 
The minimum taken between 0 and 1 simply accounts for the two possible cases, i.e., the bound based on $\#\vertices(T \cap P_m)$  with $i=0$ as in Theorem~\ref{Thbound1}, or based on  $\#\text{edges}(T \cap P_m)$ with $i=1$. A similar bound holds in the symmetric case.

\section{Applications} \label{illus}

So far, we have not provided explicit lower bounds for the nonnegative rank. As we have seen, inequalities \eqref{bound1} and \eqref{boundr3} can be interpreted as implicit lower bounds on the nonnegative rank $r_+$, but have the drawback of depending on the restricted nonnegative rank, which cannot be computed efficiently unless the rank of the matrix is smaller than 3 (Theorems~\ref{exactNMFcompl} and \ref{complRNR}). 

Nevertheless, we provide in this Section explicit lower bounds for the nonnegative rank of slack matrices (Section~\ref{SM}) and linear Euclidean distance matrices (Section~\ref{euclid}), cf.\@ introduction of Section~\ref{NR}. These bounds are derived by showing that the restricted nonnegative rank of such matrices is maximum, i.e., it is equal to the number of columns of these matrices (cf.\@ Lemma~\ref{weirdo}).

\subsection{Slack Matrices} \label{SM}


Let start with a simple observation: it is easy to construct a $m \times n$ matrix of rank $r < \min(m,n)$ with maximum restricted nonnegative rank $n$:
\begin{enumerate} 
\item Take any $(r-1)$-dimensional polytope $P$ with $n$ vertices.
\item Construct a NPP instance with $S = \vertices(P)$.
\item Compute the corresponding matrix $M$ in the equivalent RNR instance.
\end{enumerate} 
Clearly, the unique solution for NPP is $T = P = \conv(S)$ and therefore the matrix $M$ in the corresponding RNR instance must satisfy:  $\rank_+^*(M) = \#\vertices(T) = n$; see  Example~\ref{ex1} for an illustration with the three-dimensional cube. 

\begin{remark}
The matrices constructed as described above also satisfy 
\[
\rank(M) < \rank_+(M).
\] 
Otherwise $rank_+^*(M) = \rank_+(M) = \rank(M) < \min(m,n)$ which is a contradiction. This is interesting because it is nontrivial to construct matrices with $\rank(M) < \rank_+(M)$ \cite{CL10}. In fact, it is easy to check that generating randomly two nonnegative matrices $U$ and $V$ of dimensions $m \times r$ and $r \times n$ respectively, and constructing $M = UV$ will generate a matrix $M$ of rank $r$ will probability one. 
\end{remark}

In the context of compact formulations (cf.\@ Section~\ref{intro}), 
the aim is to express a polytope $Q$ with fewer constraints by using some additional variables, i.e., find a \emph{lifting} of polynomial size. 
A possible way to do that is to compute a nonnegative factorization of the slack matrix $S_M$ of $Q$ \cite{Y91} (see Equation~\eqref{SMdef}).  
The next theorem states that the restricted nonnegative rank of any slack matrix $S_M \in \mathbb{R}^{f \times v}_+$ is maximum ($f$ is the number of facets of $Q$, $v$ its number of vertices),  i.e., $\rank_+^*(S_M) = v$. This is directly related to the above observation: the slack matrix of a polytope $Q$ corresponds to a NPP instance where $Q$ is the outer polytope 
and its vertices are the points defining the inner polytope. Notice that the restricted nonnegative rank used as an upper bound for the nonnegative rank is useless in this case. 




%

\begin{theorem} \label{thSM}
Let $Q = \{ x \in \mathbb{R}^{q} \, | \, Fx \geq h, Ex = g \}$ be a $p$-dimensional polytope with $v$ vertices, $v > 1$, and let $S_M(Q)$ be its slack matrix, then $\rank_+^*(S_M(Q)) = v$. 
\end{theorem} 
\begin{proof}
In order to prove this result, we first construct a bijective transformation $L$ between $Q$ and a full-dimensional polytope $P \subset \mathbb{R}^p$. The vertices of $P$ can then be easily  constructed from the vertices of $Q$, which allows to show that $P$ and $Q$ share the same slack matrix. Finally, using the result of Theorem~\ref{equiv}, we show that the slack matrix of $P$ has maximum restricted nonnegative rank. \\

Since $Q$ is a $p$-dimensional polytope, there exists a polytope $P \subset \mathbb{R}^p$ and a bijective affine transformation 
\[
L \, : \, Q \rightarrow P : x \rightarrow 
L(x) = Ax + b 
\quad \text{ and } \quad
L^{-1} \, : \, P \rightarrow Q \, : \, y \rightarrow 
L^{-1}(y) = A^{\dag}y - A^{\dag}b, 
\]
such that $P = L(Q)$ and $Q = L^{-1}(P)$ (where $A \in \mathbb{R}^{p \times q}$ has full rank, $A^{\dag} \in \mathbb{R}^{q \times p}$ is its right inverse and $b \in \mathbb{R}^{p}$). \\
By construction, 
\begin{eqnarray*}
P & = & \{ y \in \mathbb{R}^{p} \, | \, y = L(x), x \in Q   \} 
    =   \{ y \in \mathbb{R}^{p} \, | \, L^{-1}(y) \in Q  \}, \\
 & = & \{ y \in \mathbb{R}^{p} \, | \, F L^{-1}(y) \geq h,  E L^{-1}(y) = g \}, \\
  & = & \{ y \in \mathbb{R}^{p} \, | \, F A^{\dag} y \geq h+FA^{\dag}b \},
\end{eqnarray*}
since the equalities $E L^{-1}(y) = g$ must be satisfied for all $y \in \mathbb{R}^{p}$ since $P$ is full-dimensional. 

Noting $C = F A^{\dag}$ and $d = h+FA^{\dag}b$, we have $P = \{ y \in \mathbb{R}^{q} \, | \, C y \geq d \}$. Finally, we observe that
\begin{enumerate}
\item Noting $v_i$'s the $v$ vertices of $Q$, we have that $L(v_i)$'s define the $v$ vertices of $P$. This can easily be checked since $L$ is bijective ($\forall y \in P, \exists! x \in Q \, \text{ s.t. } \,  y = L(x)$ and vice versa). 
\item $P$ can be taken as the outer polytope of a NPP instance, i.e., $P$ is bounded and $(C \; d)$ is full rank. $P$ is bounded since $Q$ is. $C$ is full rank because $P$ has at least one vertex ($v>1$). If $(C \; d)$ was not full rank, then $\exists z \in \mathbb{R}^p$ such that $d = Cz$, implying that $z \in P$.   
Since $P$ has at least two vertices ($v>1$), $\exists y \in P$ with $y \neq z$, and one can check that $y + \alpha {(y-z)} \in P$ $\forall \alpha \geq 0$. This is a contradiction because $P$ is bounded. 

\item The slack matrix of $P$ is equal to the slack matrix of $Q$: 
\begin{eqnarray*}
S_M(P) & = & CL(V)-[d \, \dots \, d] = F A^{\dag} L(V) - [h+FA^{\dag}b \, \dots \, h+FA^{\dag}b] \\
 			 & = & F (A^{\dag} L(V) - [A^{\dag}b \, \dots \, A^{\dag}b ]) - [h \, \dots \, h] \\
			 & = & F L^{-1}(L(V)) - [h \, \dots \, h] = F V - [h \, \dots \, h] \\ 
			 & = & S_M(Q),
\end{eqnarray*}
where $V = [ v_1 \, v_2 \dots v_v]$ is the matrix whose columns are the vertices of $Q$, and $L(V) = [ L(v_1) \, L(v_2) \dots L(v_v)]$ is the matrix whose columns are the vertices of $P$. 

\item The NPP instance with $P$ as the outer polytope and its $v$ vertices $L(v_i)$'s as the set of points $S$ defining the inner polytope has a unique and optimal solution $T = P = \conv(S)$ with $v$ vertices. The matrix $M$ in the  RNR instance corresponding to this NPP instance is given by the slack matrix $S_M(P)$ of $P$ implying that its restricted nonnegative rank is equal to $v$ (cf.\@ Theorem~\ref{equiv}). 
\end{enumerate}
We can conclude that $\rank_+^*(S_M(Q)) = v$.
\end{proof}

We can now derive a lower bound on the nonnegative rank of a slack matrix and on the size of an extended formulation, by combining Theorem~\ref{Thbound1} (cf.\@ Equation~\eqref{bound1}), Theorem~\ref{allbounds}, Theorem~\ref{thSM} and the result of Yannakakis \cite{Y91} (see also Section~\ref{intro}).  
\begin{corollary} \label{Goeimproved}
Let $P$ be a polytope with $v$ vertices and let $S_M \in \mathbb{R}^{f \times v}_+$ be its slack matrix of rank $r$ (i.e., $P$ has dimension $r-1$), then
\begin{equation} \label{rnnS} 
v  
\leq \phi(r,r_+) = \phi_r(r_+) 
\leq \max_{r \leq r_u \leq r_+} \binom{r_+}{r_u-r+1} 
\leq \binom{r_+}{\lfloor r_+/2 \rfloor} 
\leq   2^{r_+}, 
\end{equation}
where $r_+ = \rank_+(S_M)$. Therefore, the minimum size $s$ of any extended formulation of $P$ follows 
\[
s = \Theta(r_+ + n) \geq \Theta( \phi_r^{-1}(v) ) \ge \Theta(\log_2(v)), 
\]
where $\phi_r^{-1}(\cdot)$ is the inverse of the nondecreasing function $\phi_r(\cdot)=\phi(r,\cdot)$. 
\end{corollary} 
The last bound $2^{r_+}$ from Equation~\eqref{rnnS} is the one of Goemans \cite[Theorem 1]{G09} (see introduction of Section~\ref{NR}), and therefore Corollary~\ref{Goeimproved} provides us with an improved lower bound, even though it is still in $\Omega(\log_2(v))$. It is actually  not possible to provide an unconditionally better bound (i.e., without making additional hypothesis on the polytope $P$): since Goemans showed that the size of any LP formulation of the permutahedron (with $v=n!$ vertices) must be in   $\Omega(n\log(n))$, this implies that the nonnegative rank of its slack matrix is in  $\Omega(n\log(n))$.




\subsection{Linear Euclidean Distance Matrices} \label{euclid}

Linear Euclidean distance matrices (linear EDM's) are defined by
\begin{equation} \label{LEM}
M(i,j) = (a_i-a_j)^2, \quad 1 \leq i, j \leq n, \text{ for some } a \in \mathbb{R}^n.
\end{equation} 
In this section we assume $a_i \neq a_j \; i \neq j$, so that these matrices have rank three. Linear EDM's were used in \cite{BL09} to show that the nonnegative rank of a matrix with fixed rank (rank $3$ in this case) can be made as large as desired (while increasing the size of the matrix), implying that an upper bound for the nonnegative rank of a matrix based only on the rank cannot exist. 

We refer the reader to \cite{KW10} and the references therein for detailed discussions about Euclidean distance matrices, and related applications.

\subsubsection{Restricted Nonnegative Rank of Linear Euclidean Distance Matrices}

We first show that the restricted nonnegative rank of linear EDM's is maximum, i.e., it is equal to their dimension $n$.

\begin{definition}
The columns of a matrix $M$ have disjoint sparsity patterns if and only if 
\[
s_i \; \nsubseteq \; s_j, \quad \forall i \neq j,
\] 
where $s_i = \{ k | M(k,i) = 0 \}$ is the sparsity pattern of the $i^{\text{th}}$ column of $M$. 
\end{definition}
\begin{theorem} \label{rnrLEM}
Let $M$ be a rank-three nonnegative square matrix of dimension $n$ whose columns have disjoint sparsity patterns, then 
\[
\rank_+^*(M) = n. 
\]
In particular, linear EDM's have this property.
\end{theorem}
\begin{proof}
Let $P$, $S$ and $T$ be the polygons defined in the two-dimensional NPP instance corresponding to the RNR instance of $M$ (cf.\@ Theorem~\ref{equiv}). Aggarwal et al.\@ \cite{ABOS89} observe that if two points in $S$ are on different edges of $P$, they define a polygon with the boundary of $P$ (see each dark regions in Figure~\ref{eucill}) which must contain a point of the solution $T$. Otherwise these two points could not be contained in $T$ (see also Section~\ref{r3m}). 
\begin{figure*}[ht!]
\begin{center}
\includegraphics[width=7cm]{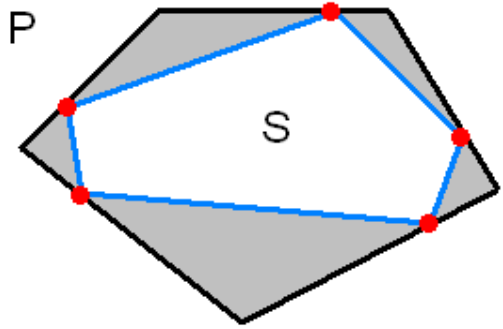}
\caption{Illustration of the restricted nonnegative rank of a linear EDM of dimension $5$. The solution $T$ must contain a point in each dark region, that is $\rank_+^*(M) = |T| = |S| = 5$.}
\label{eucill}
\end{center}
\end{figure*}
Therefore if each point of $S$ is on a different edge of the boundary of $P$, any solution $T$ to NPP must have at least $|S| = n$ vertices since $S$ defines $n$ disjoint  polygons with the boundary of $P$.  
Finally, two points $x_1$ and $x_2$ in $S$ are on different edges of the boundary of the polytope $P = \{\, x \in \mathbb{R}^{2} \, | \, Cx + d \geq 0\}$ if and only if $(Cx_1 + d)$ and $(Cx_2 + d)$ have disjoint sparsity patterns or, equivalently, if and only if the two corresponding columns of $M$ (which are precisely equal to $Cx_1 + d$ and $Cx_2 + d$) in the RNR instance have disjoint sparsity patterns. 
Indeed, for two vertices $a$ and $b$ to be located on different edges, one needs at least 
(1) one inequality that is active at $a$ and inactive at $b$ and
(2) another inequality that is active at $b$ and inactive at $a$. This is equivalent to requiring the sparsity patterns of the corresponding columns of the slack matrix to be disjoint. 
\end{proof}

\begin{remark} \label{maxRNR}
This result does not hold for higher rank matrices. For example, the matrix
\[
M  =  
\left( \begin{array}{cccccc}
     0   &  1  &   4   & 9   & 16 &   25\\
     2  &   0  &   1 &    4   &  9  &  16\\
     8   &  1  &   0 &    1   &  4  &   9\\
     13  &   4  &   1  &   0   &  1  &   4\\
    17   &  9   &  4  &   1   &  0  &   1\\
    25  &  16   &  9  &   4    & 1  &   0
\end{array} \right) 
  =  UV, \text{ with } U = \left( \begin{array}{ccccc}
		    0   &  0  &   4  &   5   &  1 \\
        1   &  0  &   1  &   3   &  0\\
        4   &  0   &  0   &  1    & 1\\
        4   &  1   &  0  &   0   &  1\\
        1   &  3  &   1  &   0   &  0\\
        0   &  5  &   4  &   0   &  1
      \end{array} \right), V = 
     \left( \begin{array}{cccccc}
				 2   &  0   &  0  &   0   &  0   &  1\\
         5  &   3  &   1  &   0   &  0 &     0\\
         0  &   0  &   1  &   1   &  0  &   0\\
         0  &   0  &   0  &   1   &  3   &  5\\
         0   &  1  &   0  &   0   &  1   &  0    \end{array} \right),
\]
has $\rank(M) = 4$ and $\rank_+^*(M) \leq 5$ since $\rank(U) = 4$. Therefore we cannot conclude that higher dimensional Euclidean distance matrices have maximal restricted nonnegative rank. 
\end{remark}

\subsubsection{Nonnegative Rank of Linear Euclidean Distance Matrices}

Since linear EDM's are rank-three symmetric matrices, one can combine the results of Theorem~\ref{rnrLEM} with Corollary~\ref{cor3} (cf.\@ Equation~\eqref{boundr3}) and  Corollary~\ref{corolrp} in order to obtain lower bounds for the nonnegative rank of linear EDM's. 
\begin{corollary} For any linear Euclidean distance matrix $M$, we have 
\begin{eqnarray*}
\rank_+^*(M) = n & \leq & \max_{3 \leq r_u \leq r_+-1} \; \min_{i=0,1} \; \faces(r_+,r_u-1,r_u-r+i)\\ 
& \leq & \max_{3 \leq r_u \leq r_+-1} \faces(r_+,r_u-1,r_u-r) = \phi'(r,r_+) \\
& \leq & \binom{r_+}{\lfloor r_+/2 \rfloor} \\
& \leq & 2^{r_+}. 
\end{eqnarray*}
 \end{corollary}
We observe that our results (first two inequalities above, from Theorem~\ref{Thbound1} and Corollary~\ref{cor3}) strengthen the bounds from Equations \eqref{beas} (Beasley and Laffey~\cite{BL09}) and \eqref{goem} (Goemans~\cite{G09}). Figure~\ref{ub3} displays the growth of the different bounds, and Table~\ref{boundNR} compares the lower bounds on the nonnegative rank for small values of $n$. 
\begin{figure*}[ht!]
\begin{center}
\includegraphics[width=12cm]{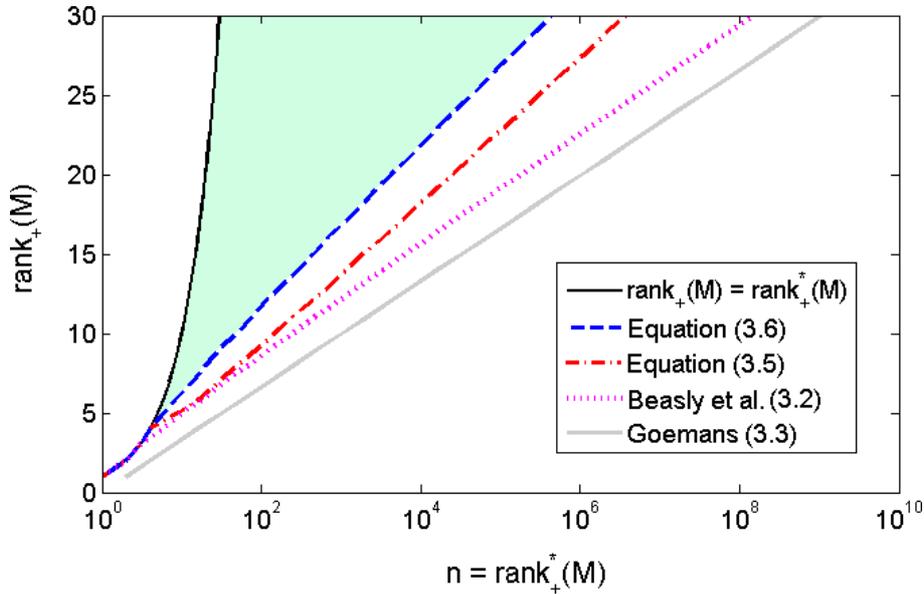}
\caption{Comparison of the different bounds for symmetric $n$-by-$n$ matrices, with $\rank_+^*(M) = n$.}
\label{ub3}
\end{center}
\end{figure*}
\begin{table}[ht!]
\begin{center}
\begin{tabular}{|c|ccccccc|}
\hline
dimension $n$  &         4 &5  & 6 & 7 & 8 & 9 & 10 \\ \hline
Equation~\eqref{boundr3} & 4 &5  & 5 & 6 & 6  & 6 & 7 \\
Equation~\eqref{bound1} &  4 &5  & 5 & 5 & 5  & 5 & 6 \\
Beasly and 
    Laffey \eqref{beas} &4 &4  & 4 & 5 & 5  & 5 & 5 \\
Goemans \eqref{goem}    &3 &3  & 3 & 3 & 4  & 4 & 4\\
\hline
\end{tabular}
\caption{Comparison of the lower bounds for the nonnegative rank of linear EDM's.}
\label{boundNR}
\end{center}
\end{table}
For example, for a linear EDM to be guaranteed to have nonnegative rank 10, the bounds requires respectively $n = 50$ \eqref{boundr3}, $n = 150$ \eqref{bound1}, $n = 252$ \eqref{beas} and  $n = 1024$ \eqref{goem}. 
This is a significant improvement, even though all the bounds are still of the same order with $r_+ \in  \Omega(\log(n))$. \\

Is it possible to further improve these bounds? Beasley and Laffey \cite{BL09} conjectured that the nonnegative rank of linear EDM's is maximum, i.e., it is equal to their dimension. Lin and Chu \cite[Theorem 3.1]{CL10} claim to have proved that this equality always holds, which cannot be correct because of the following example\footnote{In their proof, they actually show that the restricted nonnegative rank is maximum (not the nonnegative rank), see Theorem~\ref{rnrLEM}. In fact, they only consider the case when the vertices of the solution $T$ (corresponding to the columns of $U$) belong to the low-dimensional affine subspace defined by $S$ (corresponding to the column of $M$) in the NPP instance.}. 

\begin{example} \label{ex3} 
Taking 
$M \in \mathbb{R}^{6 \times 6}_+$ with 
\[
M(i,j) = (i-j)^2, \quad 1 \leq i, j \leq 6,
\] 
gives $\rank_+(M) = 5$. In fact, 
\begin{eqnarray} 
M   =   
\left( \begin{array}{cccccc}
     0   &  1  &   4   & 9   & 16 &   25\\
     1  &   0  &   1 &    4   &  9  &  16\\
     4   &  1  &   0 &    1   &  4  &   9\\
     9  &   4  &   1  &   0   &  1  &   4\\
    16   &  9   &  4  &   1   &  0  &   1\\
    25  &  16   &  9  &   4    & 1  &   0
\end{array} \right) 
  & = &  \left( \begin{array}{ccccc}
		  5     &  0  &   4  &   0    &  1 \\
      3     &  0  &   1  &  1    &  0\\
      1     &  0   &  0   & 4     & 1\\
      0     &  1   &  0  &  4    &  1\\
      0      &  3  &   1  & 1     &  0\\
      0     &  5  &   4  &  0    &  1
      \end{array} \right)
     \left( \begin{array}{cccccc}
        0   &  0   &  0    & 1    & 3   &  5\\
         5  &   3  &   1 &    0  &   0&     0\\
         0  &   0  &   1  &   1   &  0  &   0\\
          1 &    0 &    0 &    0  &   0  &   1\\
         0   &  1  &   0  &   0   &  1   &  0    \end{array} \right), \label{cex}\\
& = &  \left( \begin{array}{ccccc}
    5    &  0  &  1  &  0   &      0\\
    3   & 0    &     0  &  1    &     0\\
    1   &  0    &     0    &0   & 1\\
    0   &  1   &      0    &     0   & 1\\
    0   & 3    &     0     &    1   &      0\\
    0   &   5   & 1     &    0  &       0
          \end{array} \right)
     \left( \begin{array}{cccccc}
     0 &        0    &     0  &  1 &   3  &  5\\
        5  &  3  &  1   &      0      &   0    &     0\\
    0  &  1  &  4  &  4   & 1 &   0\\
         1  &  0  &  1 &   1  &  0   & 1\\
    4   & 1  &  0     &    0  &  1  &  4 \end{array} \right), \label{cex3}
\end{eqnarray}
so that $\rank_+(M) \leq 5$, and $\rank_+(M) \geq 5$ is guaranteed by Equation~\eqref{boundr3}, see Table~\ref{boundNR} with $\rank_{+}^*(M) = n = 6$ (or by Lemma~\ref{rMrU}, see Example~\ref{ex2}). 
\end{example} 

Example~\ref{ex3} proves that linear EDM's do not necessarily have a nonnegative rank equal to their dimension. In fact, we can even show that
\begin{theorem} \label{nrldm}
Linear EDM's of the following form 
\[
M_n(i,j) = (i-j)^2 \quad 1 \leq i,j \leq n,
\]
satisfy 
\[
\rank_+(M_n) \leq  2+  \Big\lceil \frac{n}{2} \Big\rceil ,
\]
where $\lceil x \rceil$ is the smallest integer greater or equal to $x$. 
\end{theorem}
\begin{proof}
Let first assume that $n$ is even and define 
\[
U = 
\left( \begin{array}{cc|c} 
n-1    & 0 &  \\
n-3    & 0 & \\ 
\vdots & \vdots &  I_{n/2} \\
3  & 0 & \\
1  & 0 & \\
\hline
0  & 1  & \\
0 & 3 & \\
\vdots & \vdots &  P_{n/2} \\
0 & n-3 & \\
0 & n-1 & \\
\end{array} \right),
V = 
\left( \begin{array}{cc|c} 
0    & n-1 &  \\
0    & n-3 & \\ 
\vdots & \vdots &  M_{n/2} \\
0  & 3 & \\
0  & 1 & \\
\hline
1  & 0  & \\
3 & 0 & \\
\vdots & \vdots & P_{n/2} M_{n/2} \\
n-3 & 0 & \\
n-1 & 0 & \\
\end{array} \right)^T,
\]
where $I_{m}$ is the identity matrix of dimension $m$ and $P_{m}$ is the permutation matrix with $P_m(i,j) = I_m(i,m-j+1)$ $\forall i,j$; see Equation~\eqref{cex3} for an example when $n = 6$. One can check that
\[
M_n = UV = 
\left( \begin{array}{cc} 
M_{n/2} & A + P_{n/2} M_{n/2} \\
A^T + P_{n/2} M_{n/2} &  M_{n/2} 
\end{array} \right), \text{ with }
A = \left( \begin{array}{c} 
n-1 \\
n-3 \\
\vdots\\
3 \\
1 
\end{array} \right)
\left( \begin{array}{c} 
1 \\
3 \\
\vdots\\
n-3 \\
n-1 
\end{array} \right)^T. 
\]
If $n$ is odd, we simply observe that  $\rank_+(M_n) \leq \rank_+(M_{n+1}) \leq 2+  \frac{n+1}{2} = 2+  \lceil \frac{n}{2} \rceil$, since $M_n$ is a submatrix of $M_{n+1}$~\cite{CR93}. 
\end{proof}

\begin{remark}
In the construction of Theorem~\ref{nrldm}, one can check that $\rank(V) = 4$ and the factorization can then be interpreted as a nested polytopes problem (corresponding to $M^T = V^T U^T$) in which the outer polytope has (only) dimension 3. Therefore, there is still some room for improvement and $\rank_+(M_n)$ is probably (much?) smaller. 

This example also demonstrates that, in some cases, the structure of small size nonnegative factorizations (in this case, the one from Example~\ref{ex3}) can be generalized to larger size nonnegative factorization problems. This might open new ways to computing large nonnegative factorizations. 
\end{remark}

In Example~\ref{ex3}, the nonnegative rank is smaller than the restricted nonnegative rank because there exists a higher dimensional polytope with only 5 vertices whose convex hull encloses the 6 vertices defined by the columns of $M$. Nested polytopes instance corresponding to the RNR instance with $M$ given by Example~\ref{ex3} and the two above solutions are illustrated on Figures~\ref{euc3} and~\ref{euc356} respectively (note that they are transposed to each other, but correspond to different solutions of the NPP instance), see Section~\ref{geoNN}. 
Notice that the second solution (Figure~\ref{euc356}) completely includes the outer polytope $P$; therefore, the nonnegative rank of any nonnegative matrix with the same column space as the matrix $M$ will be at most 5. 
\begin{figure*}[ht!]
\begin{center}
\includegraphics[width=13cm]{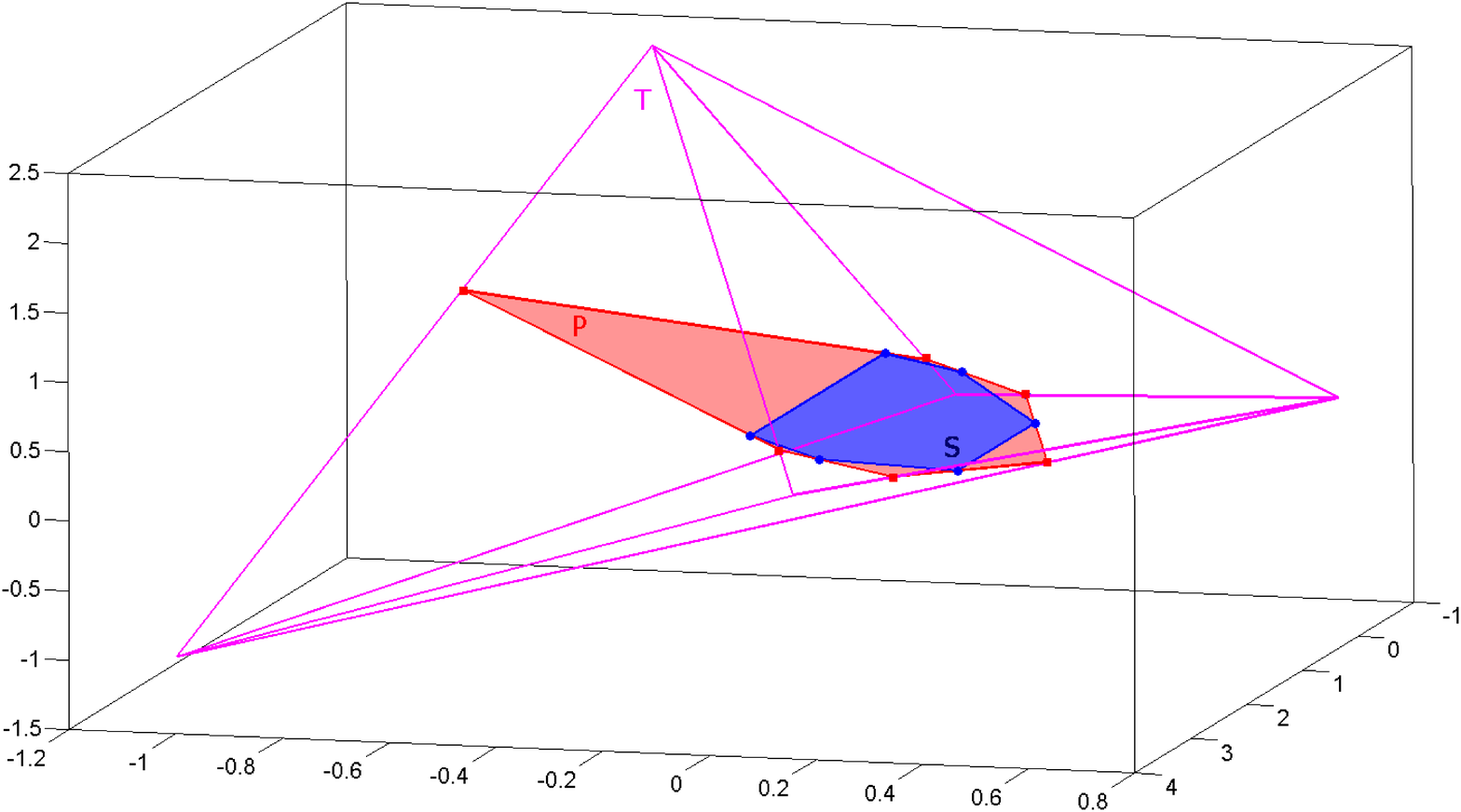}
\caption{Illustration of the solution from Equation~\eqref{cex3} as a nested polytopes problem, based on a linear EDM with $\rank(M) = 3 < \rank(U) = 4 < \rank_+(M) = 5 < \rank_+^*(M) = 6 = n$.}
\label{euc356}
\end{center}
\end{figure*}

The solutions of the above nonnegative rank problem have been computed with standard nonnegative matrix factorization algorithms \cite{LS99, CZA07} and, in general, the optimal solution is found after 10 to 100 restarts of these algorithms\footnote{These algorithms are based on standard nonlinear optimization schemes (rescaled gradient descent and block-coordinate descent), and require initial matrices $(U,V)$, which were randomly generated.}. 

We also observed than when the vectors $a$ in Equation~\eqref{LEM} used to construct the linear EDM's are chosen randomly, the nonnegative rank seems to be maximal (i.e., equal to the dimension of the matrix). In fact, even with 1000 restarts of the NMF algorithms with several random linear EDM's (of dimensions up to $n=12$, and using a factorization rank of $n-1$), every stationary point we could obtain had an error ($=\sum_{ij} (M-UV)_{ij}^2$) bounded away from zero. The following related question is still open: 
\begin{question}
Does there exist a nonnegative (symmetric?) $n \times n$ square matrix $M$ such that $\rank(M) = 3$ and $\rank_+(M) = n$, for each $n \geq 6$?
\end{question}
Table~\ref{boundNR} implies that linear EDM's with $n \leq 5$ satisfy this property\footnote{Recall we assumed $a_i \neq a_j \, \forall i \neq j$ so that such linear EDM's have rank three \cite{BL09}.}.\\

We adapt the conjecture of Beasley and Laffey \cite{BL09} as follows: 
\begin{conjecture}
Random linear EDM's of dimension $n$ are such that $\rank(M) = 3$ and $\rank_+(M) = n$ with probability one. 
\end{conjecture}

\subsection{The Nonnegative Rank of a Product}


Beasley and Laffey \cite{BL09} proved that for $A = BC$ with $A, B$ and $C \geq 0$ 
\[
\rank_+(A) \leq \rank(B) \rank(C). 
\]
In particular, $\rank_+(A^2) \leq \rank(A)^2$. They also conjectured that for a nonnegative $n \times n$ matrix $A$,
\[
\rank_+(A^2) \leq \rank(A),
\]
which we prove to be false with the following counterexample (based on a circulant matrix)
\begin{equation} \label{cexA2}
A =  \left( \begin{array}{cccccccc}
         0 &    1  &  a  &  1+a  &  1+a  &   a &1 & 0 \\
         0 &   0  &   1  &  a  &  1+a  &   1+a &a & 1\\
    		1 &   0  &   0   &  1  &  a  &   1+a &1+a & a \\
    		a &   1  &  0   &  0   &  1  &   a &1+a &  1+a\\
    		1+a &   a  &  1  &  0    &  0  &   1 &a &  1+a \\
   		1+a   &   1+a  &  a  &  1    &  0  &   0& 1 & a\\
    		a   &   1+a  &  1+a  &  a    &  1  &   0&0 & 1 \\
   		1   &   a  &  1+a  &  1+a    &  a  &   1 & 0 & 0
   \end{array} \right), 
\end{equation}
where $a=1+\sqrt{2}$. 
In fact, one can check that $\rank(A) = 3$ and $\rank_+(A^2) = 4$: indeed, $\rank_+^*(A^2) = 4$ can be computed with the algorithm of Aggarwal et al.\@ \cite{ABOS89} (see Figure~\ref{octagon} for an illustration) and, by Corollary~\ref{nr5}, $\rank_+(A^2) = \rank_+^*(A^2)$ since  $\rank_+^*(A^2) \leq \rank(A^2)+1 = 4$.  
\begin{figure*}[ht!]
\begin{center}
\includegraphics[width=8cm]{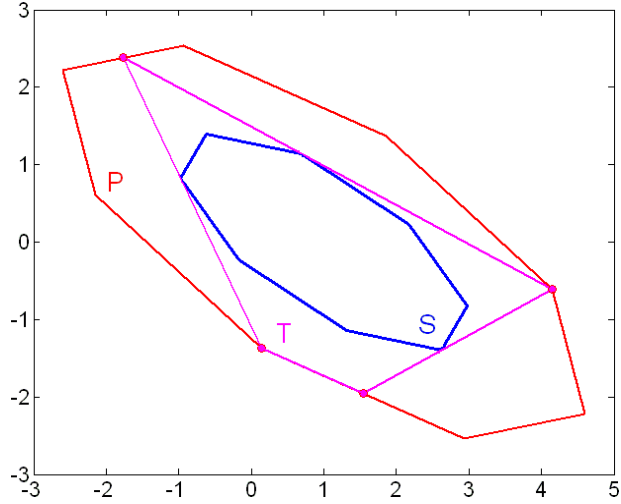}
\caption{Illustration of a NPP instance corresponding to $A^2$ and an optimal solution $T$, cf.\@  Equation~\eqref{cexA2}. See Appendix~\ref{2Drep} for the code used to perform the reduction.} 
\label{octagon}
\end{center}
\end{figure*}

\begin{remark}
The matrix $A$ from Equation~\eqref{cexA2} is the slack matrix of a regular octagon with sides of length $\sqrt{2}$. By Theorem~\ref{thSM}, we have $\rank_+^*(A) = 8$. Notice also that $A$ has rank 3 and its columns have disjoint sparsity patterns so that $\rank_+^*(A) = 8$ is implied by Theorem~\ref{rnrLEM} as well. What is the nonnegative rank of $A$? Defining 
\[
R =  \left( \begin{array}{cccccccc}
         1 &    1  &  0  &  0  &  0  &   0 &0 & 0 \\
         0 &   1  &   1  &  0  &  0  &   0 &0 & 0\\
    		0 &   0  &   1   &  1  &  0  &   0 &0 & 0 \\
    		0 &   0  &  0   &  1   &  1  &   0 &0 &  0\\
    		0 &   0  &  0  &  0    &  1  &   1 &0 &  0 \\
   		  0 &   0  &  0  &  0    &  0  &   1& 1& 0\\
    		0   &   0  &  0  &  0    &  0  &   0&1 & 1 \\
   		1   &   0 &  0  &  0    &  0 &   0 & 0 & 1
   \end{array} \right), 
\]
we have that $B = AR$ is symmetric, has rank 3 and only has zeros on its diagonal. By Theorem~\ref{rnrLEM}, $\rank_+^*(B) = 8$. Using Table~\ref{boundNR}, we have  $\rank_+(B) \geq 6$. Moreover
\[
\rank_+(AR) \leq \min(\rank_+(A),\rank_+(R)),
\]  
implying that  $6 \leq \rank_+(B) \leq \rank_+(A)$. Finally, $\rank_+(A) = 6$ because 
\begin{equation}
A = UV = 
 \left( \begin{array}{cccccc}
		    1   &  0  &   0  &   1   &  0  & a\\
        a   &  0  &   0  &   0   &  1 & a+1\\
        1   &  1   &  0   &  0    & 0 & a\\
        0   &  a-1   &  1  &   0   &  0 & 1\\
        0   &  1  &   a  &   0   &  1 &0\\
        1   &  0  &   a+1  &   0   &  a &0\\
        0   &  0  &   a  &   1   &  1 &0\\
        0   &  0  &   1  &   a-1   &  0 &1\\
      \end{array} \right) 
     \left( \begin{array}{cccccccc}
				 0   &  0   &  0  &   1   &  0   &  0 &  1 &  0\\
         1  &   0  &   0  &   0   &  0 &     1  &  a &  a\\
         1  &   1  &   0  &   0   &  0  &   0 &  0   &  0\\
         0  &   1  &   a  &   a   &  1   &  0 &  0   &  0\\
         0   &  0  &   1  &   0   &  0   &  0  &  0   &  1\\
         0   &  0  &   0  &   0   &  1   &  1  &  0   &  0
           \end{array} \right), \label{octaSol}
\end{equation}
with $\rank(U) = 4$ and $\rank(V) = 5$. 
Figure~\ref{octagonA} displays the corresponding nested polytopes problem, see Section~\ref{geoNN} and Appendix~\ref{3Drep}.  
\begin{figure*}[ht!]
\begin{center}
\includegraphics[width=\textwidth]{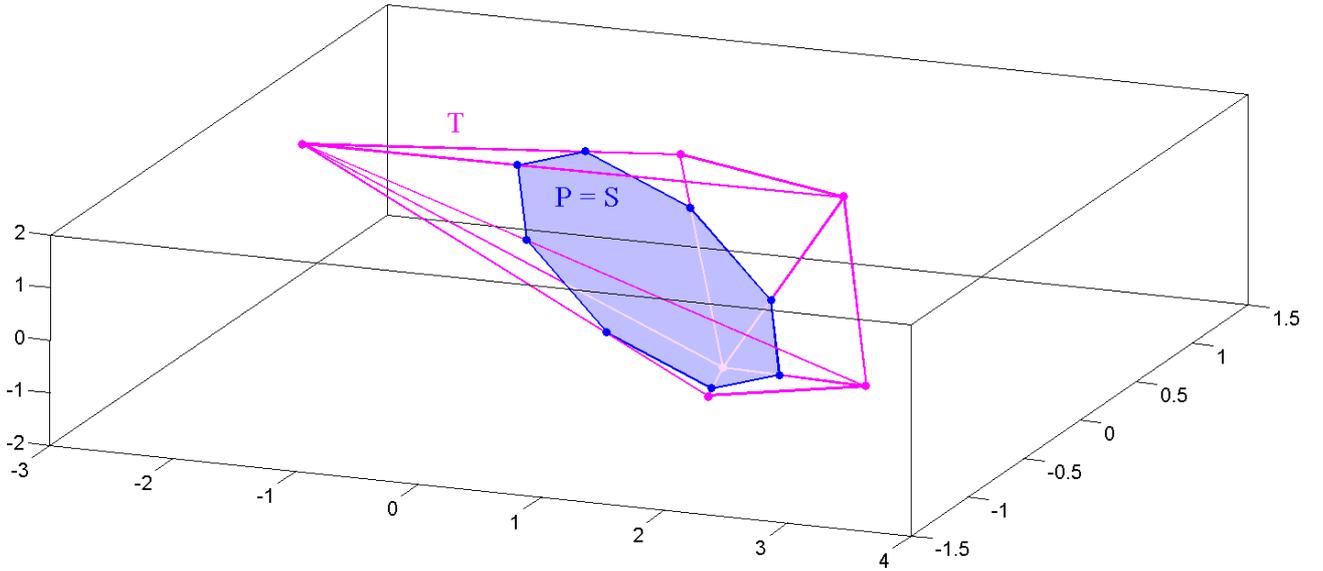}
\caption{Illustration of a nested polytopes instance corresponding to $A$ and an optimal solution, cf.\@  Equations~\eqref{cexA2} and~\eqref{octaSol}.} 
\label{octagonA}
\end{center}
\end{figure*}

It is interesting to observe that, from this nonnegative factorization, one can obtain an extended formulation (lifting) $Q$ of the regular octagon $P = \{ x \in R^2 \ | \ Cx \leq d \}$, defined as  $Q = \{ (x,y) \in \mathbb{R}^2 \times \mathbb{R}^6  \ | \ Cx + Uy = d, y \geq 0 \}$, with 
\[
C = \left( \begin{array}{cccccccc}
1  &  {\sqrt{2}}/{2}    &     0   &-{\sqrt{2}}/{2}  & -1  & -{\sqrt{2}}/{2}  &    0 &   {\sqrt{2}}/{2} \\
0   & {\sqrt{2}}/{2}  &  1  &  {\sqrt{2}}/{2}      &   0   & -{\sqrt{2}}/{2}  & -1  & -{\sqrt{2}}/{2}
           \end{array} \right)^T,
\]
and $d(i) = 1+\frac{\sqrt{2}}{2} \ \forall i$, see Equation~\eqref{extQ}. Since the system of equalities $Cx + Uy = d$ only defines 4 linearly independent equalities ($\rank([C \ U]) = 4$), the description of $Q$ can then be simplified and expressed with 4 variables and 6 inequality constraints. 

This extended formulation is actually a particular case of a construction proposed by Ben-Tal and Nemirovski \cite{BTN01} to find an extended formulation of size $\mathcal{O}(k)$ for the regular $2^k$-gon in two dimensions.
\end{remark}


\section{Concluding Remarks}

In this paper, we have introduced a new quantity called the restricted nonnegative rank, whose computation amounts to solving a problem in computational geometry consisting of finding a polytope nested between two given polytopes. 
This allowed us to fully characterize its computational complexity (see Table~\ref{complexNN}). This geometric interpretation and the relationship between the nonnegative rank and the restricted nonnegative rank also let us derive new improved lower bounds for the nonnegative rank, in particular for slack matrices and linear Euclidean distance  matrices. This also allowed us to provide counterexamples to two conjectures concerning the nonnegative rank. \\

We conclude the paper with the following conjecture: 
\begin{conjecture}
Computing the nonnegative rank and the corresponding nonnegative factorization of a nonnegative matrix is NP-hard when the rank of the matrix is fixed and greater or equal to 4 (or even possibly 3). 
\end{conjecture}
In fact, we have shown that computing a nonnegative factorization amounts to solving a nested polytopes problem in which the outer polytope might live in a higher dimensional space. Moreover, this space is not known a priori (we just know that it contains the columns of the matrix to be factorized, cf.\@ Section~\ref{geoNN}). Therefore, it seems plausible to assume that this problem is at least as difficult than the restricted nonnegative rank computation problem in which the outer polytope lives in the same low-dimensional space and is known. 
Moreover, even in the rank-three case, even though the inner polytope has dimension two, the outer polytope might have any dimension (up to the dimensions of the matrix; see, e.g., Figures~\ref{euc3} and \ref{euc356}); therefore, it seems that the nonnegative rank computation might also be NP-hard if the rank of the matrix is three. Notice that, when $\rank_+^*(M) \leq 5$, Equation~\eqref{bound1} implies $\rank_+(M) = \rank_+^*(M)$ so that the nonnegative rank can be computed in polynomial-time in this particular case.

Table~\ref{complexNN} recapitulates the complexity results for the restricted nonnegative rank and the nonnegative rank of a nonnegative matrix $M$. 
\begin{table}[ht!]
\begin{center}
\begin{tabular}{|c|c|c|}
\hline
$r = \rank(M)$ 	    &  $r_+^* = \rank_+^*(M)$ 				& 		$r_+ = \rank_+(M)$       \\
\hline
$r$ not fixed     & NP-hard     & NP-hard \cite{Vav} \\
$r \geq 4$ fixed     & NP-hard  (Theorem~\ref{complRNR})  & NP-hard? \\
$r = 3$   			  & polynomial (Theorem~\ref{exactNMFcompl}) 
			       & polynomial if $r_+^* \leq 5$  \\
   			  &  & otherwise NP-hard? \\
$r \leq 2$   			& trivial ($=r$) & trivial ($=r$) \cite{Tho} \\
\hline
\end{tabular}
\caption{Complexity of restricted nonnegative rank and nonnegative rank computations.}
\label{complexNN}
\end{center}
\end{table}

\subsection*{Acknowledgments}

We thank Mathieu Van Vyve and Didier Henrion for helpful discussions and advice. We also thank Santanu Dey for pointing our attention to the paper of Yannakakis \cite{Y91}.

\small 
\bibliographystyle{siam}
\bibliography{nnrank}

\begin{thebibliography}{10}

\bibitem{ABOS89}
{\sc A.~Aggarwal, H.~Booth, J.~O'Rourke, and S.~Suri}, {\em {Finding minimal
  convex nested polygons}}, Information and Computation, 83(1) (1989),
  pp.~98--110.

\bibitem{BL09}
{\sc L.B. Beasley and T.J. Laffey}, {\em {Real rank versus nonnegative rank}},
  Linear Algebra and its Applications, 431(12) (2009), pp.~2330--2335.

\bibitem{BTN01}
{\sc A.~Ben-Tal and A.~Nemirovski}, {\em On polyhedral approximations of the
  second-order cone}, Mathematics of Operations Research, 26(2) (2001),
  pp.~193--205.

\bibitem{BP73}
{\sc A.~Berman and R.J. Plemmons}, {\em {Rank factorization of nonnegative
  matrices}}, SIAM Review, 15(3) (1973), p.~655.

\bibitem{BBLPP07}
{\sc M.W. Berry, M.~Browne, A.N. Langville, V.P. Pauca, and R.J. Plemmons},
  {\em Algorithms and applications for approximate nonnegative matrix
  factorization}, Computational Statistics and Data Analysis, 52 (2007),
  pp.~155--173.

\bibitem{BC96}
{\sc J.~Bhadury and R.~Chandrasekaran}, {\em {Finding The Set of All Minimal
  Nested Convex Polygons}}, in Proceedings of the 8th Canadian Conference on
  Computational Geometry, 1996, pp.~26--31.

\bibitem{B94}
{\sc J.W. Boardman}, {\em Geometric mixture analysis of imaging spectrometry
  data}, in Proc. IGARSS 4, Pasadena, Calif., 1994, pp.~2369--2371.

\bibitem{CR10}
{\sc E.~Carlini and F.~Rapallo}, {\em {Probability matrices, non-negative rank,
  and parameterization of mixture models}}, Linear Algebra and its
  Applications, 433 (2010), pp.~424--432.

\bibitem{CL08}
{\sc M.T. Chu and M.M. Lin}, {\em {Low-Dimensional Polytope Approximation and
  Its Applications to Nonnegative Matrix Factorization}}, SIAM J. Sci. Comput.,
  30(3) (2008), pp.~1131--1155.

\bibitem{CZA07}
{\sc C.~Cichocki, R.~Zdunek, and S.~Amari}, {\em {Hierarchical ALS Algorithms
  for Nonnegative Matrix and 3D Tensor Factorization}}, Lecture Notes in
  Computer Science, Springer, 4666 (2007), pp.~169--176.

\bibitem{C90}
{\sc K.L. Clarkson}, {\em {Algorithms for polytope covering and
  approximation}}, in Proceedings of the Third Workshop on Algorithms and Data
  Structures, 1993, pp.~246--252.

\bibitem{CR93}
{\sc J.E. Cohen and U.G. Rothblum}, {\em {Nonnegative ranks, Decompositions and
  Factorization of Nonnegative Matrices}}, Linear Algebra and its Applications,
  190 (1993), pp.~149--168.

\bibitem{CCCZ09}
{\sc M.~Conforti, G.~Cornuéjols, and G.~Zambelli}, {\em Extended formulations
  in combinatorial optimization}, 4OR: A Quarterly Journal of Operations
  Research, 10(1) (2010), pp.~1--48.

\bibitem{C94}
{\sc M.D. Craig}, {\em Minimum-volume tranforms for remotely sensed data}, IEEE
  Transactions on Geoscience and Remote Sensing, 32(3) (1994), pp.~542--552.

\bibitem{gau}
{\sc G.~Das}, {\em Approximation schemes in computational geometry}, PhD
  thesis, University of Wisconsin-Madison, 1990.

\bibitem{DG97}
{\sc G.~Das and M.~Goodrich}, {\em {On the Complexity of Optimization Problems
  for Three-Dimensional Convex Polyhedra and Decision Trees}}, Computational
  Geometry: Theory and Applications, 8 (1997), pp.~123--137.

\bibitem{DJ90}
{\sc G.~Das and D.A. Joseph}, {\em {The Complexity of Minimum Convex Nested
  Polyhedra}}, in Proc. of the 2nd Canadian Conference on Computational
  Geometry, 1990, pp.~296--301.

\bibitem{CGP81}
{\sc D.~de~Caen, D.A. Gregory, and N.J. Pullman}, {\em The boolean rank of
  zero-one matrices}, in Proc. 3rd Caribbean Conference on Combinatorics and
  Computing, pp.\@ 169-173, 1981.

\bibitem{DS03}
{\sc D.~Donoho and V.~Stodden}, {\em {When does non-negative matrix
  factorization give a correct decomposition into parts?}}, in In Advances in
  Neural Information Processing 16, 2003.
\newblock MIT Press.

\bibitem{GG08}
{\sc N.~Gillis and F.~Glineur}, {\em {Nonnegative Factorization and The Maximum
  Edge Biclique Problem}}.
\newblock \textsc{CORE} Discussion paper 2008/64, 2008.

\bibitem{GP10}
{\sc N.~Gillis and R.J. Plemmons}, {\em {Dimensionality reduction,
  classification, and spectral mixture analysis using nonnegative
  underapproximation}}, in SPIE conference Volume 7695, paper 46, Orlando,
  2010.

\bibitem{G09}
{\sc M.X. Goemans}, {\em {Smallest compact formulation for the permutahedron}}.
\newblock Talk at ISMP, Chicago, 2009.

\bibitem{IC99}
{\sc A.~Ifarraguerri and C.-I. Chang}, {\em Multispectral and hyperspectral
  image analysis with convex cones}, IEEE Transactions on Geoscience and Remote
  Sensing, 37(2) (1999), pp.~756--770.

\bibitem{KCD09}
{\sc B.~Klingenberg, J.~Curry, and A.~Dougherty}, {\em Non-negative matrix
  factorization: Ill-posedness and a geometric algorithm}, Pattern Recognition,
  42(5) (2009), pp.~918--928.

\bibitem{KW10}
{\sc N.~Krislock and H.~Wolkowicz}, {\em {Euclidean Distance Matrices and
  Applications (survey)}}, tech. report, University of Waterloo, 2010.
\newblock CORR 2010-2010-06.

\bibitem{LS99}
{\sc D.D. Lee and H.S. Seung}, {\em {Learning the Parts of Objects by
  Nonnegative Matrix Factorization}}, Nature, 401 (1999), pp.~788--791.

\bibitem{LS09}
{\sc T.~Lee and A.~Shraibman}, {\em Lower Bounds in Communication Complexity},
  Foundations and Trends in Theoretical Computer Science, 2009.

\bibitem{CL10}
{\sc M.M. Lin and M.T. Chu}, {\em {On the nonnegative rank of Euclidean
  distance matrices}}, Linear Algebra and its Applications, 433(3) (2010),
  pp.~681--689.

\bibitem{MS92}
{\sc J.S.B. Mitchell and S.~Suri}, {\em {Separation and Approximation of
  Polyhedral Surfaces}}, Operations Research Letters, 11 (1992), pp.~255--259.

\bibitem{O77}
{\sc J.~Orlin}, {\em Contentment in graph theory: Covering graphs with
  cliques}, Indagationes Mathematicae (Proceedings), 80(5) (1977),
  pp.~406--424.

\bibitem{Peet}
{\sc R.~Peeters}, {\em {The maximum edge biclique problem is NP-complete}},
  Discrete Applied Mathematics, 131(3) (2003), pp.~651--654.

\bibitem{Tho}
{\sc L.B. Thomas}, {\em {Rank factorization of nonnegative matrices}}, SIAM
  Review, 16(3) (1974), pp.~393--394.

\bibitem{Vav}
{\sc S.A. Vavasis}, {\em On the complexity of nonnegative matrix
  factorization}, SIAM Journal on Optimization, 20 (2009), pp.~1364--1377.

\bibitem{W91}
{\sc C.A. Wang}, {\em {Finding Minimal Nested Polygons}}, BIT, 31 (1991),
  pp.~230--236.

\bibitem{W06}
{\sc V.~Watts}, {\em {Fractional biclique covers and partitions of graphs}},
  Electronic journal of combinatorics, 13 (2006), p.~\#R74.

\bibitem{Y91}
{\sc M.~Yannakakis}, {\em {Expressing Combinatorial Optimization Problems by
  Linear Programs}}, Computer and System Sciences, 43 (1991), pp.~441--466.

\bibitem{Z95}
{\sc G.M. Ziegler}, {\em Lectures on Polytopes}, Springer-Verlag, 1995.

\end{thebibliography}

\appendix

\normalsize

\section{MATLAB Codes}

In this Appendix, codes for two specific reductions are provided: 
\begin{enumerate}
\item[A.1] The reduction from any RNR instance of a rank-three matrix to a two-dimensional NPP instance.
\item[A.2] The geometric interpretation and the visualization of a nonnegative factorization  $M = UV$ when $\rank(M) = 3$ and $\rank(U) = 4$ as a solution of a nested polytopes problem, where the inner polytope $S$ is two-dimensional and the outer polytope $P$ is three-dimensional (see Section~\ref{geoNN}). 
\end{enumerate}

\subsection{RNR to NPP when $\rank(M) = 3$}  \label{2Drep}

The following code has been used to generate Figure~\ref{octagon}.

\begin{lstlisting}
% 2-D Representation of a NPP instance corresponding to the RNR instance 
% of a rank-3 nonnegative matrix M, cf. Theorem 2.
%
% [P,A,B] = NMFrank3(M)
%
% Input.
%   M              : (m x n) matrix or rank 3.
% Output.
%   P (2 x p<=m)         :  vertices of the outer simplex defined with Cx+d>=0.
%   A (m x 3), B (3 x n) :  M = AB and columns of A and B sum to one. 

function [P,A,B] = NMFrank3(M)

if rank(M) > 3 || min(M(:)) < 0
    disp('The matrix is not rank 3 or not nonnegative'); return;
end
k = 3; [m,n] = size(M);
% 1. Remove zero row/columns and normalize the columns of M
M = M(sum(M')>0,sum(M)>0); D = diag(1./sum(M)); M = M*D;
% 2. Compute the decompositon M = AB
[A,B] = basesumtoone(M,3);
% 3. Find the inequalities of the set P = { x in R^{k-1} | Cx+d >= 0 }
C = A(:,1:k-1)-repmat(A(:,k),1,k-1); d = A(:,k);
% 4. Draw P (outer polytope) and S (inner polytope)
P = vertices(C,d); K = convhull(P(1,:),P(2,:));
figure; plot(P(1,K),P(2,K),'ro'); hold on; plot(P(1,K), P(2,K),'r');
K = convhull(B(1,:),B(2,:)); plot(B(1,:),B(2,:),'bo'); plot(B(1,K),B(2,K),'b-');
\end{lstlisting}

\begin{lstlisting}
% Compute a rank-k decomposition of M = AB such that columns of A and B sum to one
function [A,B] = basesumtoone(M,k)

[u,s,v] = svds(M,k); A = u*s; B = v';
sA = sum(A); 
if min(sA) < 1e-3
    A(:,2:k) = A(:,2:k) + repmat(A(:,1),1,k-1); 
    B(1,:) = B(1,:)- sum(B(2:k,:)); 
    sA = sum(A); 
    A = A*diag(1./sA); B = diag(sA)*B;
else
    A = A*diag(1./sA); B = diag(sA)*B;
end
\end{lstlisting}

\begin{lstlisting}
% Find vertices V of the set P = { x in R^{k-1} | Cx+d >= 0 } with brute force
function V = vertices(C,d);

[m,k] = size(C);
V = []; lP = 0;
choices = nchoosek(1:length(C(:,1)),k);
% Choose two inequations of Cx+d >= 0 and compute the intersection
for i = 1 : length(choices(:,1))
    if rank(C(choices(i,:),:)) == k
        x = C(choices(i,:),:)\[-d(choices(i,:))];
        % Check if the intersection is in P
        if  min(C*x+d) >= -1e-9 && (lP == 0 || min(sum((V-repmat(x,1,lP)).^2))>1e-6)
            V = [V x]; lP = lP+1;
        end
    end
end
\end{lstlisting}

\subsection{3-D representation} \label{3Drep}

The following code has been used to generate Figures~\ref{euc3}, \ref{euc356} and \ref{octagonA}.

\begin{lstlisting}
% 3-D Representation of a nonnegative factorization of M = UV 
% with rank(M) = 3 and rank(U) = 4. Displays only the intermediate simplex 
% T and the set of (inner) points S, cf. Section 3.1. 
%
% [A,B,Bp] = Visualisation3D(M,U)
%
% Input.
%   M>=0  (m x n)   : M is a rank 3.
%   U>=0  (m x k)   : U is a rank 4, and s.t. there exists V >=0: M = UV.
% Output.
%   A (m x 4), B (4 x n) : U = AB and columns of A and B sum to one. 
%              Bp(4 x n) : M = ABp and columns of Bp sum to one. 

function [A,B,Bp] = Visualisation3D(M,U)

if rank(U) ~= 4 || min(U(:)) < 0  || rank(M) ~= 3 || min(M(:)) < 0
    disp('The matrix U (resp. M) is not rank 4 (resp. 3) or not nonnegative'); return;
end
% 0. Columns of M and U sum to one
D = diag(1./sum(M)); M = M*D; Du = diag(1./sum(U)); U = U*Du;
% 1. Compute U = AB
[A,B] = basesumtoone(U,4);
% 2. Display the columns of B and draw T
P = B; K = convhulln(P(1:3,:)'); figure;
for i = 1 : length(K(:,1))
    plot3(P(1,K(i,1:2)),P(2,K(i,1:2)),P(3,K(i,1:2)),'m','linewidth',2); hold on;
    plot3(P(1,K(i,2:3)),P(2,K(i,2:3)),P(3,K(i,2:3)),'m','linewidth',2);
    plot3(P(1,K(i,[1 3])),P(2,K(i,[1 3])),P(3,K(i,[1 3])),'m','linewidth',2);
end
% 3. Compute and display the columns of Bp (M = ABp) and draw S
Bp = A\M;  K = convhull(Bp(1,:),Bp(2,:));
plot3(Bp(1,K), Bp(2,K), Bp(3,K),'bo','linewidth',2); hold on;
plot3(Bp(1,K), Bp(2,K), Bp(3,K),'b-','linewidth',2);
\end{lstlisting}


\end{document}